\documentclass[10pt]{amsart}
\usepackage{amssymb, amscd, amsmath, amsthm, latexsym}
\usepackage{pb-diagram,  graphicx, psfrag}

\def\ff{{\bf F}}

\def\rr{{\bf R}}
\def\co{\colon\thinspace}
\def\cs{\mathop{\#}}

\def\cala{\mathcal{A}}
 \def\calc{\mathcal{C}}
\def\calm{\mathcal{M}}
\def\calg{\mathcal{G}}
\def\calt{\mathcal{C}_{TS}}

\newcommand{\fig}[2] { \includegraphics[scale=#1]{#2} }

\newcommand{\nbdk}{\nu K}
\def\bull{\vbox{\hrule\hbox{\vrule\kern3pt\vbox{\kern6pt}\kern3pt\vrule}\hrule}}

\newcommand{\cm}{\cdot}
\newcommand{\F}{\mathbb F}
\newcommand{\Z}{\mathbb Z}
\newcommand{\A}{\mathbb A}
\newcommand{\B}{\mathbb B}
\newcommand{\C}{\mathbb C}
\newcommand{\Q}{\mathbb Q}

\newcommand\relspinc{\underline{\spinc}}

\newcommand\x{\mathbf x}

\newcommand\ModFlow{\mathcal M}

\newcommand\ModSphere{\ModFlow\left({\mathbb S}\longrightarrow 
\Sym^{g-1}(\Sigma_{1})\times \Sym^2(\Sigma_{2})\right)}
\newcommand\ModSpheres\ModSphere
\newcommand\CF{CF}

\newcommand\CFa{\widehat{CF}}
\newcommand\CFp{\CFb}
\newcommand\CFm{\CF^-}

\newcommand\CFb{CF^+}

\newcommand\UnparModSp{\widehat \ModSp}
\newcommand\UnparModFlow\UnparModSp
\newcommand\Mod\ModSp

\newcommand{\cald}{{\mathcal D}}

\newcommand\spin{\mathfrak s}

\newcommand{\spinc}{\ifmmode{{\mathfrak s}}\else{${\mathfrak s}$\ }\fi}

\newcommand{\spinct}{\mathfrak t}

\newcommand\ModMaps{\mathcal M}
\newcommand\ModSp\ModMaps

\newcommand\Ta{{\mathbb T}_{\alpha}}

\newcommand\Tb{{\mathbb T}_{\beta}}

\newcommand\alphas{\mbox{\boldmath$\alpha$}}

\newcommand\betas{\mbox{\boldmath$\beta$}}

\newcommand\spincrel\relspinc

\newcommand\CFK{CFK}

\newcommand\CFKa{\widehat\CFK}

\newcommand\SpinC{\mathrm{Spin}^c}

\newcommand\Dual{\mathcal D}
\newcommand\Duality\Dual

\newcommand\ons{Ozsv{\'a}th and Szab{\'o}}
\newcommand\os{Ozsv{\'a}th-Szab{\'o}}

\def\zz{{\mathbb Z}}

\def\ff{{\mathbb F}}
 
\def\qq{{\mathbb Q}}

\def\rr{{\mathbb R}}
\def\co{\colon\thinspace}
\def\cs{\mathop{\#}}
\def\caln{\mathcal{N}}

\def\cala{\mathcal{A}}
\def\calc{\mathcal{C}}
\def\calm{\mathcal{M}}
\def\calg{\mathcal{G}}

\def\co{\colon}
\def\Sp{Spin$^c$}
\DeclareMathOperator{\cfk}{\it CFK}
\DeclareMathOperator{\hfk}{\it HFK}
\DeclareMathOperator{\hf}{\it HF}

\DeclareMathOperator{\Wd}{w}

\newtheorem{thm}{Theorem}
\newtheorem{theorem}{Theorem}[section]
\newtheorem{lemma}[theorem]{Lemma}
\newtheorem{corollary}[theorem]{Corollary}
\newtheorem{prop}[theorem]{Proposition}
\theoremstyle{definition}
\newtheorem{definition}[theorem]{Definition}
\def\co{\colon\thinspace}

\numberwithin{equation}{section}

\begin{document}

\title{Topologically slice knots of smooth concordance order two}

\author{Matthew Hedden}\author{Se-Goo Kim}\author{Charles Livingston}

\address{Matthew Hedden: Department of Mathematics, Michigan State University, East Lansing, MI 48824 }
\email{mhedden@math.msu.edu}

\address{Se-Goo Kim: Department of Mathematics and Research Institute for Basic Sciences, Kyung Hee University, Seoul 130-701, Korea }
\email{sgkim@khu.ac.kr}

\address{Charles Livingston: Department of Mathematics, Indiana University, Bloomington, IN 47405 }
\email{livingst@indiana.edu}

\thanks{This work was supported in part by the National Science Foundation under Grants  0707078, 0906258, CAREER 1150872,  and 1007196, by an Alfred P. Sloan Research Fellowship,  and by the National Research Foundation of Korea grant funded by the Korea government (MEST) NRF-2011-0012893.\\ \today}

%%%%%%%ABSTRACT%%%%%%%%%%%%%%

 \begin{abstract}  The existence of topologically slice knots that are of infinite order in the knot concordance group followed from  Freedman's work on topological surgery and    Donaldson's gauge theoretic approach to four-manifolds.  Here, as an application of Ozsv\'ath and Szab\'o's Heegaard Floer theory, we show the existence of an infinite subgroup of the smooth concordance group generated by  topologically slice knots of concordance order two.  In addition, no nontrivial element in this subgroup can be represented by a knot with Alexander polynomial one.  \end{abstract}

\maketitle

%%%%%%%SECTION%%%%%%%%%%%%%%

\section{Introduction.}\label{sectionintroduction}

In~\cite{fm} Fox and Milnor defined the smooth knot concordance group $\calc$.  Their proof that $\calc$ is infinite quickly yields an infinite family of distinct elements of order two.    Results of Murasugi~\cite{mura} and Tristram~\cite{tris} demonstrated that $\calc$ also contains a free summand of infinite rank.  This work culminated in Levine's construction~\cite{le} of a surjective homomorphism $\phi\co \calc \to \calg$, where $\calg$ is an algebraically defined group isomorphic to the infinite direct sum $\zz^\infty \oplus \zz_2^\infty \oplus \zz_4^\infty$.

Classical surgery theory allowed Levine to  prove that $\phi$ is an isomorphism in high (odd) dimensions.  The first distinction between classical and  high-dimensional concordance was seen in the work of Casson and Gordon~\cite{cg}, who showed  that the kernel of $\phi$ is nontrivial; this was followed by a proof by Jiang~\cite{ji} that $\ker(\phi)$ contains a subgroup isomorphic to $\zz^\infty$.  In~\cite{liv2} it was shown that $\ker(\phi)$ also contains a subgroup isomorphic to $\zz_2^\infty$.
 
The work  of Donaldson~\cite{d} and Freedman \cite{fr, freedman-quinn} on smooth and topological $4$--manifolds,  respectively, revealed further subtlety present in low-dimensional concordance.   One can  define a concordance group $\calc^{top}$ in the topological, locally flat, category. The distinction between the smooth and topological categories is highlighted by considering the kernel  of the natural surjection $\calc  \to \calc^{top}$.  This kernel is generated by topologically slice knots, and we denote it $\calt$.   To underscore the importance of $\calt$ it should be mentioned that a single non-trivial element in $\calt$ implies the existence of a smooth $4$--manifold homeomorphic, but not diffeomorphic, to $\rr^4$.  Several people, including Akbulut and Casson, observed that   the results  of Donaldson and Freedman  can be used to produce non-trivial elements in $\calt$ (see~\cite{cochrangompf}),  but until recently little was known about the structure of $\calt$.     Using   techniques developed by Donaldson~\cite{d} and later enhanced by Fintushel-Stern \cite{FS} and Furuta \cite{Furuta},  Endo~\cite{e}   proved that $\calt$ contains a subgroup isomorphic to $\zz^\infty$ (see also \cite{hk,hlr, Hom} for other constructions of infinite rank free subgroups).   Techniques derived from Heegaard Floer theory and Khovanov homology (specifically the Rasmussen invariant~\cite{RasSlice}) were used to show that $\calt$ contains a {\it summand} isomorphic to $ \zz ^3$~\cite{liv3,  liv4, mo}. Recently that work has been superseded by work of Hom~\cite{Hom3} which applies a deep analysis of the structure of   Heegaard Floer complexes to  construct a summand isomorphic to $\zz^\infty$.   

With the abundance of $2$--torsion in $\calc$, one might expect that $\calt$ likewise has such torsion.  However, producing torsion classes in $\calt$ is  quite difficult since one needs a manifestly smooth invariant to detect them.  Many of the known techniques for analyzing $\calt$, however,  fail at detecting torsion classes   (for instance, ~the \os\ \cite{FourBall} or Rasmussen \cite{RasSlice} concordance invariants).   Our main result shows that  like the concordance group,  $\calt$ has an abundance of $2$--torsion.

\begin{thm}\label{maintheorem} {\it $\calt$ contains a subgroup isomorphic to $\zz_2^\infty$.}
\end{thm}
\vskip.1in

  We conjecture that, in line with Hom's result,  a  summand isomorphic to $\zz_2^\infty$  exists, but current tools  seem insufficient to prove this.

Freedman's work~\cite{fr, freedman-quinn} implied that all knots of Alexander polynomial one are topologically slice, and these knots provided all the early examples of nontrivial elements in $\calt$.  However, in~\cite{hlr} it was shown that $\calt$ in fact contains a subgroup isomorphic to $\zz^\infty$ with no nontrivial element represented by a knot with Alexander polynomial one.  Here we extend this to 2--torsion.   Let $\calc_\Delta$ denote the subgroup of $\calc$ generated by knots with Alexander polynomial one.

\begin{thm}\label{poly1thm} {\it The subgroup from Theorem \ref{maintheorem} can be chosen so that no nontrivial member is representable by a knot with Alexander polynomial one.  In particular, the group $\calc_{TS}/\calc_\Delta$ contains a subgroup isomorphic to  $\zz_2^\infty$.}
\end{thm}
This theorem can be strengthened by replacing the subgroup of knots generated by   Alexander polynomial one knots with  the subgroup generated by knots with determinant one. 
\vskip.1in

To prove these theorems  we consider  knots $K_{J,n}$ as illustrated in Figure~\ref{figureknot}.  These knots are defined to be the boundaries of surfaces built by adding two bands  to a disk as shown: the bands are tied in knots $J$ and $-J$ and have $n$ and $-n$ full twists, where $n>0$.   An important special case occurs when $U$ is the unknot, whereby $K_{U,1}$ is the figure eight knot.  We have the following easy proposition:

\begin{prop}   $K_{J,n}$ is negative amphicheiral {\em(}$K_{J,n} = - K_{J,n}${\em)}; in particular,  $2 K_{J,n} = 0 \in \calc$.  If $J_1$ and $J_2$ are concordant, then  $K_{J_1,n}$ and $  K_{J_2,n}$ are concordant.
\end{prop}

\noindent The amphicheirality of $K_{J,n}$ can be demonstrated just as for the case  $J=U$.  Indeed, an isotopy to $-K_{J,n}$ is obtained by pulling the bottom band through the rectangular region and then rotating the knot $180^\circ$ about a vertical axis running down the center of the page.  The second part of the lemma follows from the fact that satellite operations  descend to concordance, and $K_{J,n}$ is a two-fold satellite operation with companions $J$ and $-J$.  The proposition allows for the immediate construction of  elements of order at most two in $\calt$.

\begin{corollary}  For $U$ the unknot, $2(K_{J,n} \# K_{U,n}) = 0 \in \calc.$  If $J$ is topologically slice, then $ K_{J,n} \# K_{U,n}$ is topologically slice; that is,   $ K_{J,n} \# K_{U,n} \in \calt$.
\end{corollary}
 
\begin{figure}[h]
\fig{.5}{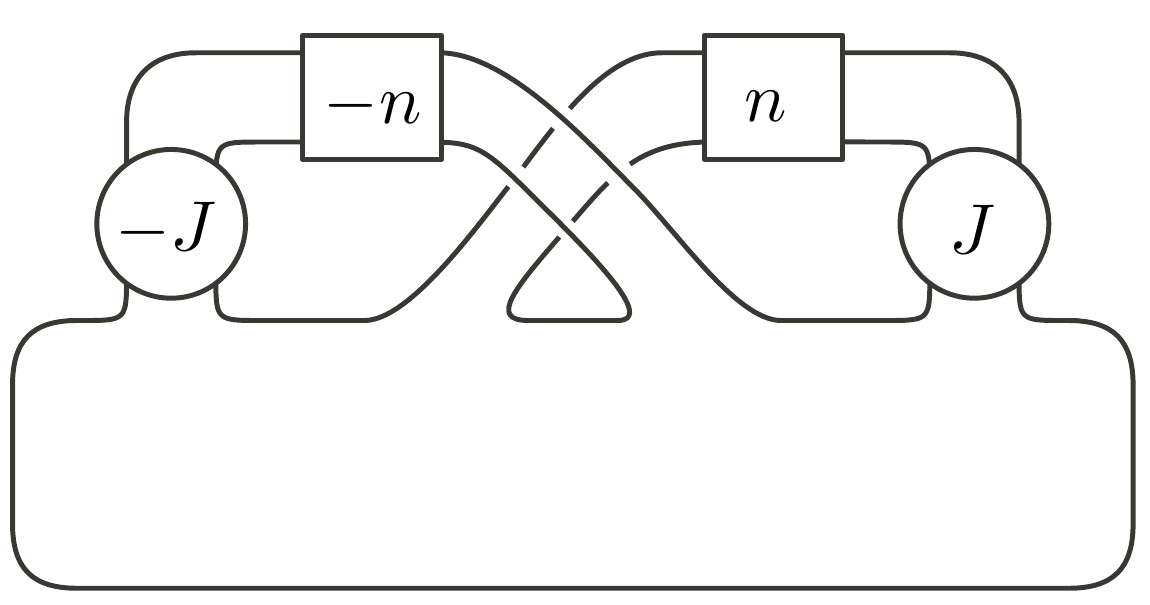}
\caption{}\label{figureknot}
\end{figure}
 
Let $D   $ denote the untwisted Whitehead double of the right-handed trefoil knot, $T_{2,3}$, and let $D_k$ denote $kD$. The knots $ K_{D_k,n} \# K_{U,n} $ provide the subgroups appearing in Theorem~\ref{maintheorem} and Theorem~\ref{poly1thm}.

\begin{thm}\label{distinctthm} There exists an infinite set of pairs of positive integers $\{(k,n)\}$ with the property that the corresponding set of  knots $\{K_{{D_k},n}\# K_{U,n}\}$ generates a subgroup of $\calc_{TS}$ and  of $\calc_{TS}/\calc_\Delta$ as described in Theorems~\ref{maintheorem} and~\ref{poly1thm}.
\end{thm}

  The proof of Theorem~\ref{distinctthm} is presented in Section~\ref{sectionexamples} after necessary background is given in Section~\ref{prelimsec}.  The proof depends on a detailed analysis of the Heegaard Floer $d$--invariants of the branched cover of $S^3$ branched over $K_{D_k, n}$.  That analysis occupies Sections~\ref{sectionhf1},~\ref{sectionhf2} and~\ref{computesection}.  Some of the most technical work has been placed in appendices.

\vskip.1in
\noindent{\bf Acknowledgements}  We are indebted to the referees for their unusually thorough and thoughtful reading of the original manuscript.  Their contribution significantly enhanced the clarity and accuracy of the presentation.
 
\vskip.2in

%%%%SECTION %%%%%%%%
\section{Preliminary constructions}\label{prelimsec}

%%%%SUBSECTION %%%%%%%%
\subsection{Algebraic slicing obstructions.}

The proofs of our main results are based on considering two-fold branched covers of $S^3$ over $ K_{J,n}$, which we denote $M(K_{J,n})$.   According to~\cite{akbulutkirby}, $M(K_{J,n})$ has a surgery description as illustrated in~Figure~\ref{figurecover}, in which the meridian $\mu$ is labeled for later reference.  In the diagram, $J^r$ denotes the orientation reverse of $J$, and the meridian of the  surgery curve is oriented consistently with a choice of orientation for that curve. 
(In general, if a link is formed from the Hopf link by tying a local knot $K_1$ in one component, $K_2$ in the second, and then performing $n_1$ and $n_2$ surgery on the link, we denote the resulting manifold $S^3_{n_1,n_2}(K_1,K_2)$.)
If $J$ is reversible,   then $M(K_{J,n})$ has the surgery description $S^3_{-2n,2n}(-2J,2J)$.

From this surgery description,  a quick calculation yields a computation of the homology of $M(K_{J,n})$.  In particular, $H_1(M(K_{J,n}))$ is a cyclic group of order $4n^2 +1$. Notice that given the choice of generator $\mu$ of $H_1(M(K_{J,n}))$, the identification with a cyclic group is canonical.  In particular, this observation along with Poincar\'e duality permits us to identify $H^2(M(K_{J,n}))$ with $H^2(M(K_{U,n}))$ for all $J$.  For emphasis and for later reference we state this as a proposition.

\begin{prop}\label{orderh1} The choice of surgery description of $M(K_{J,n})$ provides a canonical isomorphism $H^2( M(K_{J,n})) \cong H_1( M(K_{U,n})) \cong \zz_{4n^2 +1}$.
\end{prop}

\begin{figure}[h]
\fig{.5}{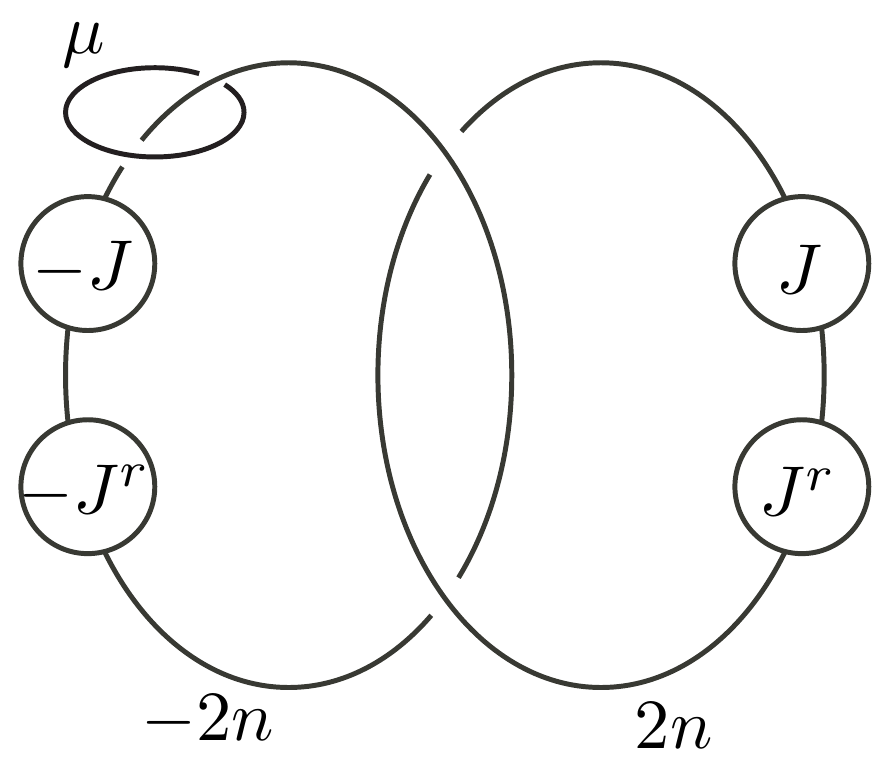}
\caption{}\label{figurecover}
\end{figure}

As a special case, we note that $M(K_{U,n})$ is given by $(4n^2 + 1)/2n$--surgery on the unknot: $M(K_{U,n}) =L(4n^2+1, 2n)$.

If a knot $K$ is slice with slice disk $F^2$, then $M(K)$ bounds the two-fold branched cover of $B^4$ branched over the slice disk, $W(F^2)$.  In this case we have the following from~\cite{cg}.

\begin{prop} The homology groups $H_i(W(F^2 ), \zz_2) = 0 $ for $i \ge 1$.  The image $I$ of the restriction map $H^2(W(F^2)) \to H^2(M(K))$ is a subgroup of order satisfying $|I|^2 = |H_1(M(K))|$. Furthermore, $I$ is self-annihilating with respect to the linking form.  (Via duality, we can view the linking form, usually defined on $H_1(M(K))$, as a form on $H^2(M(K))$.)
\end{prop}

%%%%SUBSECTION %%%%%%%%
\subsection{Slicing obstructions from Heegaard Floer theory.}\label{spinc}

Heegaard Floer theory associates a  (filtered homotopy class of) chain complex $CF^\infty(M, \spinc) $ to a 3--manifold $M$ with  \Sp\  structure $\spinc$.   For a  manifold $X$, the set of \Sp\ structures,  \Sp$(X)$, is in bijection with elements in $H^2(X)$, though not canonically so.   However, associated to each $\spinc \in \text{\Sp} (X)$,   there is a first Chern class, $c_1(\spinc)\in H^2(X)$, and in the case that $H^2(X,\zz_2) = 0$, the map: $$c_1:  \text{\Sp}(X)\rightarrow H^2(X)$$ provides a bijection that  is natural  with respect to the transitive action of $H^2(X)$ on both sides and with respect to pull-back; that is

\begin{enumerate}
\item $c_1(\spinc+\alpha)=c_1(\spinc)+2\alpha $ for all $\alpha\in H^2(X)$, and
\item $c_1(i^*\spinc)= i^*c_1(\spinc)$ for an embedding  $i:Y\rightarrow X$ with trivial normal bundle. In particular, for the inclusion of a  codimension zero submanifold $Y\subset X$, or for $Y \subset \partial X$, we have $c_1(\spinc|_Y)=c_1(\spinc)|_Y$.
\end{enumerate}
Thus, in cases in which   $H^2(X,\zz_2) = 0$, via the Chern class we can denote \Sp\ structures by $\spinc_\alpha$ for $\alpha\in H^2(X)$.  There is an involution on the set of \Sp\ structures called  conjugation; the conjugate of  $\spinc$ is denoted $\overline{\spinc}$ and  one has $\overline{\spinc_\alpha} =  \spinc_{-\alpha}$. 

As described in greater detail in Section~\ref{sectionhf1},  there is an invariant $d(M, \spinc)$, called the {\it correction term}, defined in terms of the filtered homotopy type of $CF^\infty(M, \spinc) $.  It satisfies the following properties.

\begin{enumerate} \item $d(-M,\spinc) = -d(M, \spinc)$. \vskip.1in
\item $d(M_1 \# M_2, \spinc_1 \# \spinc_2) = d(M_1  , \spinc_1 )+ d(  M_2,  \spinc_2)$.\vskip.1in
\item $d(M, \overline{\spinc} ) = d(M, \spinc)$.
\end{enumerate}

The following theorem from~\cite{os2} provides the obstruction we will use to show that knots are not smoothly slice.  (The use of $d$ as a slicing obstruction first  appeared in~\cite{mo}, where it was applied only for the Spin structure.  In~\cite{grs, jn} it was   used in conjunction with a careful analysis of \Sp\ structures to study concordance.)

\begin{prop}\label{propd} Suppose $(W,\spinct)$ is a \Sp\  four-manifold  satisfying $H_i(W,\qq) = 0, i > 0$, and $M = \partial W$. Then $d(M,  \spinct |_M) = 0$.
\end{prop}

\vskip.05in 

\noindent{\bf Note.}  In the case that $M^3$ is constructed as $-n$--surgery on an oriented knot $K\subset S^3$, there is the following enumeration of \Sp \ structures on $M$, parameterized by integers $m$ with $-n/2   \le m < n/2$ (see \cite[Section 4]{Knots} for details).  If $W$ denotes the four-ball with a two-handle added along $K$ with framing $-n<0$,   we let $\spinct_m$ denote the \Sp \ structure on $W$ satisfying $\left< c_1(\spinct_m), [S]\right>  + n = 2m$, where $[S]$ is the generator of $H_2(W)$ represented by an oriented Seifert surface for $K$, capped off with the core of the two-handle.  We denote by $\spinc_m$ the restriction of $\spinct_m$ to $M$.  This is well-defined whether $n$ is odd or even.  The Poincare dual of $c_1(\spinc_m)$ satisfies  $\text{PD}(c_1(\spinc_m)) = 2m[\mu]$, where $[\mu] \in H_1(M)$ is the class represented by the meridian of $K$.
 
 %%%%%%%SECTION%%%%%%%%%%%%%%
\section{Main theorem.}\label{sectionexamples}

In Appendix~\ref{appennumber} we use a theorem of Iwaniec to obtain a number theoretic result.  

\begin{prop}\label{defineNprop} There exists an infinite set  $\caln$ of positive integers greater than one such that  for all $n \in \caln$,  $4n^2+1$ is square free  and $4n^2 + 1$ is a product of at most two primes.  Furthermore, for each $m, n \in \caln$, $4m^2 +1 $ and $4n^2+1$ are relatively prime.

\end{prop}

 The main results of this  paper are consequences of the following theorem.

\begin{theorem}\label{main1} For each $n\in \caln$   there is a positive integer $k_n$ having  the following property:  If  $n \in \caln$ and    $L$ is any knot with $|H_1(M(L))| $ relatively prime to $4n^2 +1$, then $K_{D_{k_n}, n} \cs K_{U,n} \cs L$ is not slice.

\end{theorem}

Most important, as an immediate corollary we have the result that implies Theorems~1,~2, and~3 of the introduction.

\begin{corollary}\label{maincor}  For all nonempty finite subsets $\caln' \subset \caln$, $\sum_{n \in \caln'} (K_{D_{k_n},n}\cs K_{U,n}) \notin \calc_\Delta$.  
In particular, the set of knots $\{ K_{D_{k_n}, n} \cs K_{U,n} \}$ generate a subgroup isomorphic to $\zz_2^\infty$ in $\calc_{ts}/\calc_{\Delta}$.
\end{corollary} 

\begin{proof}[{\bf Proof} Corollary~\ref{maincor}]
Suppose that $\sum_{n \in \caln'} (K_{D_{k_n},n}\cs K_{U,n}) $ is concordant to a knot $K$ with Alexander polynomial one.   Then we have $\sum_{n \in \caln'} (K_{D_{k_n},n}\cs K_{U,n}) \cs -K$ is slice.   Let   $m$ be the least $n \in \caln'$ and let $\caln''$ be the set $\caln'$ with $m$ removed.  We can break up the connected sum of knots as
$$(K_{D_{k_{m},m}}\cs K_{U,m} ) \cs (  \sum_{n \in \caln''} (K_{D_{k_n},n}\cs K_{U,n}) \cs -K).$$  At this point we can complete the proof by  applying Theorem~\ref{main1} with $L =  \sum_{n \in \caln''} (K_{D_{k_n},n}\cs K_{U,n}) \cs -K$.
\end{proof}

\subsection{Proof of Theorem~\ref{main1}}
The rest of this section presents the proof of Theorem~\ref{main1}, calling upon results from later sections as needed. The choice of $k_n$ will be described in the context of the proof.

Abbreviate  $K_{D_{k}, n} \cs K_{U,n} $ by $K_{n,k}$.  Assuming that $K_{n,k} \cs L$ is slice, the manifold $M(K_{n,k}) \cs M(L)$ bounds a rational homology ball $W$.  Since the orders of $H_1(M(K_{n,k}))$ and $H_1(M(L))$ are relatively prime, it follows that the image of $H^2(W)$ in $H^2(M(K_{n,k})\cs M(L)) \cong ( \zz_{4n^2+1} \oplus \zz_{4n^2+1}) \oplus H^2(M(L))$ contains a subgroup of the form   $\calm \oplus 0$ where $\calm \subset   \zz_{4n^2+1} \oplus \zz_{4n^2+1}$    is a metabolizer for the linking form on $H_1(M(K_{n,k}))$. With this we can prove the following.

\begin{lemma}\label{lemma-a} If $K_{n,k} \cs L$ is slice, then for some metabolizer $\calm$ of the linking form on $H_1(M(K_{n,k}))$ and for all $(z_1, z_2) \in \calm$, $d(M(K_{D_{k},n}), \spinc_{z_1}) + d( M(K_{U,n}),\spinc_{z_2}) = 0$. 

\end{lemma} 

\begin{proof}  It is immediate that $d(M(K_{D_{k},n}), \spinc_{z_1}) + d( M(K_{U,n}),\spinc_{z_2}) + d(M(L),\spinc_0) = 0$.  Notice that since $L$ is assumed to be concordant to $-K_{n,k}$, which is of order two, $L$ is also of order 2.  Because $2L$ is slice, $2M(L)$ bounds a $\zz_2$--homology ball $Z$.  The Spin structure on  $Z$  restricts to the Spin structure on $2M(L)$.  Thus, the \Sp\ structure $\spinc_0 \oplus \spinc_0$ on $M(L) \cs M(L)$ extends to $Z$.  It follows that   $2d(M(L),\spinc_0) = 0$.  

\end{proof}

We now must consider  metabolizers for the linking form on  $(\zz_{4n^2+1})^2 $.

\begin{lemma}\label{theoremmet2} For a fixed non-degenerate linking   form on $ \zz_{N}  $, with $N$ square-free, each metabolizer for the double of this form   on $(\zz_{N})^2 $ is generated by an element $(1,b)$ where $1+b^2 \equiv 0 \mod N$.
\end{lemma}

\begin{proof} 
Recall first that a non-degenerate linking form on $\zz_N$ is given by an element $\alpha \in \zz_N$:  $lk(x,y) \equiv x\alpha y \mod N$, where $\alpha $ and $N$ are relativity prime. 

Since  $(\zz_{N})^2 $ is of rank two, any metabolizer $\calm$ is of rank at most two, so is generated by two elements, $\{(a,b), (c,d)\}$.  Using Gauss-Jordan elimination, we see it is generated by a pair of elements $\{(a,b), (0,c)\}$.  If $c$ is nonzero  it would have self-linking 0, which is impossible for a non-degenerate form on $\zz_N$ with $N$ square-free.  

Thus $\calm$ is generated by a single element $(a,b)$, so $(a,b)$ is of order $N$.   If either $a$ or $b$ were divisible by some prime factor of $N$, then some multiple of $(a,b)$ would be of the form $(0,c)$ or $(c,0)$ with $c$ nonzero.  But again, the existence of such an element is ruled out by $N$ being square-free and the form being non-degenerate.  Since $a$   must be  relatively prime to $N$, some multiple of $(a,b)$ is of the form $(1,b')$, and clearly $b' \ne 0$.  In fact, since $(1, b')$ is in the metabolizer $\calm$, one has $1 + (b')^2 \equiv 0 \mod N$, as desired.  \end{proof}

Combining Lemmas~\ref{lemma-a} and \ref{theoremmet2} yields the following.

\begin{lemma} If $K_{n,k} \cs L$ is slice, then  for some $b$ satisfying $1 + b^2 \equiv 0 \mod 4n^2+1$ and for all $x$,  $d(M(K_{D_{k},n}), \spinc_{x}) + d( M(K_{U,n}),\spinc_{bx}) = 0$.
\end{lemma}
Notice that in this statement the subscripts on  the \Sp\   structures, $x$ and $bx$, are cohomology classes;  the cohomology of the spaces are identified using Proposition~\ref{orderh1}.

For our purposes, a change of signs will be convenient, as follows.

\begin{lemma}\label{lemma-b} If $K_{n,k} \cs L$ is slice, then  there is some $b$ satisfying $ b^2  \equiv 1 \mod  4n^2+1$ such that    for all $x$,  $d(M(K_{D_{k},n}), \spinc_{x}) = d( M(K_{U,n}),\spinc_{bx}) $.

\end{lemma}

\begin{proof} The knot $K_{U,n}$ is of order two: $ K_{U,n} \cs K_{U,n}$ is slice.  Thus, the previous argument shows that there is some $b'$ satisfying $1+b'^2 \equiv 0 \mod 4n^2+1$ such that for all $x$,  $d(M(K_{U,n}), \spinc_{x}) + d( M(K_{U,n}),\spinc_{b'x}) = 0$.   Replacing $x$ with $bx$ from the previous lemma yields    $d(M(K_{U,n}), \spinc_{bx}) + d( M(K_{U,n}),\spinc_{b'bx}) = 0$.  The rest is arithmetic along with a renaming of variables.

\end{proof}

\noindent{\bf Completion of the proof of Theorem~\ref{main1}}

According to Proposition~\ref{prop:dmn}, there is a specific \Sp\ structure $\spinc_\alpha$ such that for all $k$ with $0 \le k < n/2$, 
$$d(M(K_{D_k, n}, \spinc_\alpha) )- d(M(K_{U, n}, \spinc_\alpha)) = -2k.$$  Applying Lemma~\ref{lemma-b}, for each $k$ and some $b$ satisfying $1 + b^2 \equiv 0 \mod 4n^2 +1$, we have
$$d(M(K_{U, n}), \spinc_{b \alpha} )- d(M(K_{U, n}), \spinc_\alpha) = -2k.$$  
Since $4n^2+1$ is the product of at most two primes, there are at most four values of $b \mod 4n^2 +1$ for which $  b^2 \equiv -1 \mod 4n^2 +1$.  Thus, the expression on the left of the equality can have at most four distinct values.   As long as $n\ge 9$ the number of elements in the interval $0 \le k < n/2$ contains at least five elements, so we can choose $k$ so that the equality is violated.  Any such choice can serve as $k_n$.

%%%%SUBSECTION %%%%%%%%
 
  %%%%%%%SECTION%%%%%%%%%%%%%%
\section{Heegaard Floer  complexes}\label{sectionhf1}

The computation of the $d$--invariants of interest depends upon a detailed understanding of  related Heegaard Floer complexes.  The main result in this section is Theorem~\ref{thmrefilter}, the {\it refiltering theorem}, which describes the chain complex associated to the meridian of a knot $K$ within the manifold $S^3_{-N}(K)$ in terms of the chain complex associated to $K$ within $S^3$.

%%%%SUBSECTION %%%%%%%%
\subsection{Three-manifold complexes.} \label{background} We let $\ff$ denote the field with two elements.  As mentioned earlier, given a 3--manifold $M$ with \Sp\ structure $\spinc$, there is an associated  $\Z$--filtered $\qq$--graded complex $CF^\infty(M,\spinc)$.  This complex is a free, finitely generated $\ff[U, U^{-1}]$--module, which is well-defined up to filtered chain homotopy equivalence.   The filtration of $CF^\infty(M,\spinc)$ by subcomplexes is induced by a natural filtration of $\F[U,U^{-1}]$ by powers of $U$.  More precisely, we can regard $\F[U,U^{-1}]$ as an (infinitely generated) $\F[U]$-module in the obvious way.  As such, it has an exhaustive $\Z$-indexed filtration by (free) $\F[U]$-submodules $$ ...\subset U^{k}\F[U]\subset U^{k-1}\F[U]\subset U^{k-2}\F[U]\subset ...,$$ and  this filtration induces a $\Z$-filtration of $\CF^\infty(M,\spinc)$ by subcomplexes. Thus the filtration level of a chain in  $\CF^\infty(M,\spinc)$, regarded as a sum of Laurent polynomials in the basis elements,  is given by the negative of the  minimum power of $U$ which appears in this polynomial.  The action of $U$ clearly lowers filtration level by one.   It lowers grading by two.  

Added notation permits the simple representation of subcomplexes; for instance, we denote the subcomplex consisting of elements of filtration level at most $n$ by $CF^\infty(M,\spinc)_{\{i \le n\}}$.  With this we
can define several associated complexes, $$CF^-(M,
\spinc)=CF^\infty(M,\spinc)_{\{i <0 \}}, \ \ CF^+(M,\spinc) =   CF^\infty(M,\spinc)/  CF^\infty(M,\spinc)_{\{i <0 \}},$$  and  $$\widehat{CF}(M,\spinc) =   CF^\infty(M,\spinc)_{\{i \le 0\}}/  CF^\infty(M,\spinc)_{\{i <0 \}}.$$
There are corresponding homology groups, $HF^-(M,\spinc), HF^+(M,\spinc)$ and $\widehat{HF}(M,\spinc) $.

There will also be situations in which we must shift the gradings of elements in these chain complexes.  For instance, we will write $CF^+(M,\spinc)[\epsilon]$ for the same complex as  $CF^+(M,\spinc)$, except with  the homological grading of any element increased by $\epsilon$;  that is,  $$CF^+_*(M,\spinc)[\epsilon] = CF^+_{*-\epsilon}(M,\spinc), $$
 for all $*$.
 
\begin{definition} The \emph{$d$--invariant} $d(M,\spinc)$ is given by
$$\min\{ gr(\alpha) \ | \ \alpha\ne0 \in HF^+(M,\spinc)\ \text{and}\ \alpha\in \text{Image}\ U^n\ \text{for all}\ n>0\},$$
where  $gr(\alpha)$ is the homological grading.
\end{definition}

%%%%SUBSECTION %%%%%%%%
\subsection{Knot complexes.}

 A knot $K \subset M$  induces a second $\Z$--filtration of the complex $CF^\infty(M,\spinc)$, which thus becomes a $\qq$--graded, $\Z\oplus\Z$--filtered  complex.  The $U$ action respects the second filtration, lowering this filtration by one as well.  This doubly filtered complex is denoted $CFK^\infty(M,K, \spinc)$, and again there are associated subcomplexes such as $CFK^\infty(M,K,\spinc)_{\{i \le m, j \le n \}}$.  As in the 3--manifold case,  there are quotient complexes $CFK^+(M,K, \spinc) = CFK^\infty(M,K,\spinc) /CFK^\infty(M,K,\spinc)_{\{i < 0\}}$ and $\widehat{CFK}(M,K,\spinc)= CFK^\infty(M,K,\spinc)_{\{i \le 0\}} /CFK^\infty(M,K,\spinc)_{\{i < 0\}}$.  Note that ignoring the $j$ filtration yields the corresponding complexes for $(M, \spinc)$.

Figure~\ref{figurecomplexes-knots} illustrates the complexes for the unknot and the $(2,5)$--torus knot in $S^3$.  (For alternating knots $K$, $CFK^\infty(S^3,K)$ is determined simply from the Alexander polynomial~\cite{os1}.) The dots represent elements in a filtered $\F$--basis and the  line segments indicate components of the boundary operator.  Sometimes we will not need to include arrows on the segments; the fact that the boundary map cannot increase either filtration  and $\partial^2 = 0$ will make the direction unambiguous in most of the examples we consider. The gradings are not indicated in the diagram; the coordinates in the diagram correspond to the filtration, as follows: the vertical and horizontal axes in bold separate elements of filtration levels $-1$ and $0$.  That is, the dot just above and to the right of the origin has filtration level $(0,0)$.  The action of $U$ shifts the diagram down and to the left by one.  \vskip.1in

 \noindent{\bf Convention.} In all the cases we consider, $CFK^\infty(M,K, \spinc)$  is filtered chain homotopy equivalent to  $C \otimes_\ff {\ff[U,U^{-1}]} $ for some finite $\Z\oplus\Z$--filtered $\ff$--complex $C$.  We will simplify  our diagrams and illustrate only $C$, leaving out all of its $U$ translates.

\begin{figure}[h]
\fig{.5}{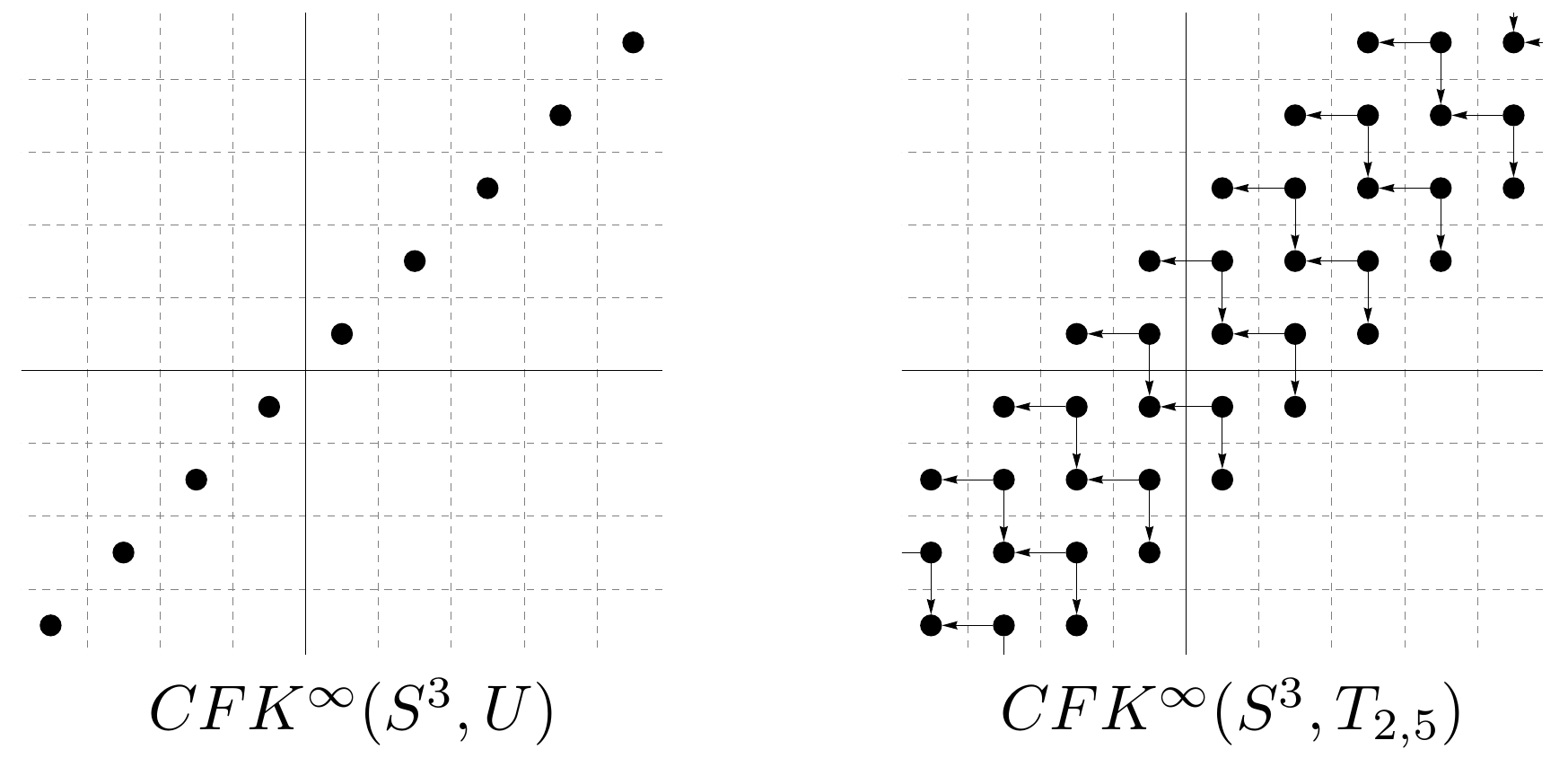}
\caption{}\label{figurecomplexes-knots}
\end{figure}

%%%%SUBSECTION %%%%%%%%
\subsection{Gradings.}
To this point we have not described how the homological grading is determined.  Rather than review this aspect of the theory, we refer the interested reader to ~\cite{os2,HolDiskFour} for definitions and details.  For our purposes, the following elementary observation will be  particularly useful:   the value of $d(M, \spinc)$ can be used to determine the gradings of elements in $CFK^\infty(M,K, \spinc)$.  We illustrate this with an important example.

In the special case of $S^3$ there is only one \Sp\ structure, denoted $\spinc_0$.    We have $HF^+(S^3, \spinc_0) = \ff[U,U^{-1}]/U\ff[U]$ and  by definition $d(S^3, \spinc_0) = 0$.  For example, in the complex $CFK^+(S^3, T_{2,5}, \spinc_0)$  (constructed from the complex illustrated on the right in Figure~\ref{figurecomplexes-knots} by quotienting by all elements to the left of the vertical axis), we see that the non-trivial homology class with least grading is represented by the cycles living in filtration levels $(0,2), (1,1)$, and $(2,0)$.  Thus, all of these have grading 0.

%%%%SUBSECTION %%%%%%%%
\subsection{Meridians of knots in surgered manifolds.}

Let $S^3_{-N}(K)$ denote the manifold constructed as $-N$ surgery on $K \subset S^3$ and let $\mu$ denote the meridian of $K$, viewed as a knot in $S^3_{-N}(K)$.  The work of~\cite{hedden1} can be extended to show that for each \Sp\ structure $\spinc_m$,  the complex $CFK^\infty(S^3_{-N}(K),\mu,\spinc_m)$ is isomorphic to $CFK^\infty(S^3,K)$, but endowed with a different $\Z\oplus\Z$--filtration and an overall shift in the homological grading. We state the result for a knot in a general $3$--manifold.
 \vskip.05in

\noindent {\bf Notation}  Notice that until now, \Sp\ structures were denoted $\spinc_\alpha$, where $\alpha \in H^2(M)$. Here they have been denoted $\spinc_m$ with $m$ an integer (viewed, modulo $N$, in $\Z/N\Z$), according to the convention described in the note at the end of Section \ref{spinc}. 
\vskip.05in

 \begin{theorem}[\bf Refiltering Theorem]\label{thmrefilter} Suppose $N \ge 2g(K)$. For $m$ in the interval $$\lceil (-N+1)/2 \rceil \le m \le \lfloor N/2
  \rfloor,$$  the complex $CFK^\infty(Y^3_{-N}(K), \mu, \spinc_m)$ is isomorphic to  $CFK^\infty(Y^3, K)[\epsilon_1]$ as an unfiltered complex, where $[\epsilon_1]$ denotes a grading shift  that depends only only on $m$ and $N$.  Given a generator $\{[x,i,j]\}$ for $CFK^\infty(Y^3, K)$, the $\Z\oplus\Z$ filtration level of the same generator, viewed as a chain in $CFK^\infty(Y^3_{-N}(K), \mu, \spinc_m)$, is given by:

$$\mathcal{F}_m([x,i,j]) = \begin{cases}
[i,i] & \text{if\ \   } j > i+m,\\
[j-m, j-m-1] & \text{if\ \   } j \le i+m.
\end{cases} $$

\end{theorem}

\noindent Before discussing its proof, we illustrate this theorem in Figure~\ref{figurerefilter}, which shows for all $N \ge 8$  the complexes $CFK^\infty(S^3_{-N}(K), \mu, \spinc_m)$ for $K = U$ and $K = -T_{2,5}$, with $-3\le m \le 4$.  We show only the $\ff$--subcomplex that generates the full  complex over $\ff[U,U^{-1}]$.

\begin{figure}[h]
     \fig{1}{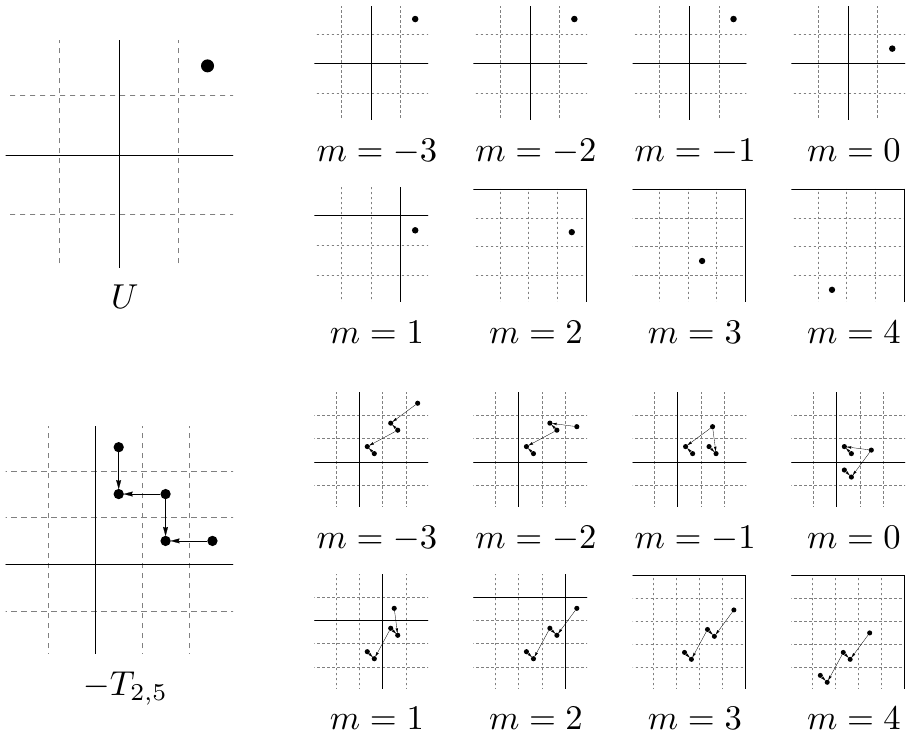}  
        \caption{}\label{figurerefilter} 
\end{figure}

\begin{proof}
The theorem refines \cite[Theorem 4.1]{hedden1} in two directions:
\begin{enumerate}
\item \cite[Theorem 4.1]{hedden1} determines the $\Z$--filtered chain homotopy type of  \newline $\CFKa(S^3_{-N}(K), \mu, \spinc_m)$.   Here we seek to understand the $\Z\oplus\Z$--filtered chain homotopy type of $CFK^\infty(Y^3_{-N}(K),\mu,\spinc_m)$.

\item \cite[Theorem 4.1]{hedden1}  applies  for $N\gg 0$.  We wish to show that $N=2g(K)$ suffices.
\end{enumerate}

The first refinement is an immediate extension of the proof from \cite{hedden1}, so we do not belabor the details here. To begin, we note that the difference between $S^3$ and a general $3$--manifold is merely notational.  The key idea from \cite{hedden1} was to observe that with the addition of another basepoint, the natural Heegaard diagram for $-N$--framed surgery on $K$ could be made to represent the knot $\mu\subset Y_{-N}(K)$.  The proof of \cite[Theorem 4.1]{Knots} shows that the $\Z$--filtered chain homotopy type of $CF^{\infty}(Y_{-N}(K),\spinc_m)$ is determined by that of $CFK^{\infty}(Y,K)$.  This implies that  the chain homotopy type of the  complexes $\CFm(Y_{-N}(K),\spinc_m)$, $\CFp(Y_{-N}(K),\spinc_m)$, and $\CFa(Y_{-N}(K),\spinc_m)$ are also determined by $CFK^{\infty}(Y,K)$, as they are sub, quotient, and subquotient complexes of the filtration, respectively. Now the meridian $\mu\subset Y_{-N}(K)$ induces an additional $\Z$--filtration of any of these complexes, and \cite[Theorem 4.1]{hedden1} determined that in the case of $\CFa(Y_{-N}(K),\spinc_m)$, the additional $\Z$--filtration consists of two steps:
$$ 0 \subseteq CFK^\infty(Y,K)_{\{i\ge 0, j= m\}}\subseteq CFK^\infty(Y,K)_{\{\text{min}(i,j-m)=0\}},$$
where the subquotient on the right was identified with $\CFa(Y_{-N}(K),\spinc_m)$ by~\cite[Theorem 4.1]{Knots}.  Strictly speaking, the proof of ~\cite[Theorem 4.1]{hedden1} only dealt with the case of  positive framed surgery explicitly, leaving the case of negative framings to the reader.  The analogous proof for negative framings yields the two-step filtration above, and   the extension to $CFK^{\infty}$ follows  easily from the same proof.\footnote{It is worth noting that the formula from ~\cite[Theorem 4.1]{hedden1} was actually for the filtration induced by  $\mu^r$, the meridian of $K$ with reversed orientation.  The formula above is for the meridian with its standard orientation.}   To be more precise, \cite[Theorem 4.1]{Knots} identifies $CF^{\infty}(Y_{-N}(K),\spinc_m)$ with $CFK^{\infty}(Y,K)$ via a chain map which was denoted $\Phi$.  This isomorphism of chain complexes respects the $\F[U,U^{-1}]$--module structure of both complexes, and hence one of the $\Z$--filtrations.  The additional  $\Z$--filtration on $CF^{\infty}(Y_{-N}(K),\spinc_m)$ induced by $\mu$ can be determined in exactly the same manner as it was determined for the case of $\CFa(Y_{-N}(K),\spinc_m)$ in \cite{hedden1}, yielding the statement of the theorem.  In both cases, the key lemma is \cite[Lemma 4.2]{hedden1}, which identifies the $\Z$--filtration induced on any given $i=$constant slice in $CF^{\infty}(Y_{-N}(K),\spinc_m)$  with a two step filtration as above. 

For the second refinement, recall that the proof of \cite[Theorem 4.1]{hedden1} relies on making the surgery parameter large enough so that an entire $\SpinC$ equivalence class of generators  for $Y_{-N}(K)$ is {\em supported  in the  winding region} (by definition, we say that a generator is supported in the winding region if it is represented by a $k$-tuple of  intersection points which contains a point in the region shown in \cite[Figure 13]{hedden1}).   This is achieved by a pigeonhole argument: there are only finitely many $\SpinC$ equivalence classes that can be represented by the finitely many generators  {\em not} supported in the winding region, and increasing $N$ increases the number of $\SpinC$ structures without bound.  Once we have an entire $\SpinC$ equivalence class supported in the winding region, we can  appeal to the technique of ``moving the basepoint."  In the present context  this means moving the placement of the meridian and nearby collection of basepoints   throughout the winding region; see \cite[Theorem 4.3]{hedden1}.    This technique allows us to use the single $\SpinC$ equivalence class of intersection points which is supported  in the winding region to represent all $|H_1(Y)|\cm N$ different $\SpinC$ structures on $Y_{-N}(K)$ (for a manifold with $b_1(Y)>0$,  $|H_1(Y)|$ should be replaced by the number of $\SpinC$ structures on $Y$ represented by the diagram).   Thus the question is reduced to finding a topological interpretation for the number of $\SpinC$ classes represented by generators which are not supported in the winding region.  We will henceforth refer to such generators as {\em exterior}.

To  achieve a bound for the number of $\SpinC$ classes represented by exterior generators,  we use a particular Heegaard diagram which is adapted to a Seifert surface for $K$ with genus $g$.  A similar Heegaard diagram appears in the proof of the adjunction inequality \cite[Theorem 5.1]{Knots}; such a diagram is constructed 
explicitly in~\cite[Lemma 7.3]{oz:hf-properties} and~\cite[Proof of Theorem 2.1]{Ni}.  The diagram consists of a quadruple, 
$$(\Sigma_k, \vec{\alpha}=\{\alpha_1,...,\alpha_k\}, \vec{\beta}=\{\beta_1,...,\beta_{k-1},\mu,\lambda\},\{w\cup z\}),$$
where $(\Sigma,\vec{\alpha},\vec{\beta}\setminus \mu)$ and $(\Sigma,\vec{\alpha},\vec{\beta}\setminus \lambda)$ are Heegaard diagrams for $Y_0(K)$ and $Y$, respectively, and $\{w\cup z\}$ specifies $K$ on the latter diagram.  The key features of the diagram are that 
\begin{itemize}
\item There is a  domain $\mathcal{P}$ with $\partial{\mathcal{P}}=\alpha_k\cup \lambda$ such that  $\mathcal{P}\cup \{\text{Disk bounded by} \ \alpha_k\}$ is isotopic to the chosen Seifert surface. 
\item The only $\alpha$ curves which intersect $\mathcal{P}$ are $\alpha_k$ and $\alpha_1,...,\alpha_{2g}$, where $g$, as above, is the genus of $K(=\text{genus of}\ \mathcal{P})$.
\end{itemize}  Now we observe that the diagram
$$(\Sigma_k, \vec{\alpha}=\{\alpha_1,...,\alpha_k\}, \vec{\beta}=\{\beta_1,...,\beta_{k-1},\lambda^{-N}\},\{w\})$$
specifies $Y_{-N}(K)$, where $\lambda^{-N}$ is a simple closed curve isotopic to the resolution of $N$ parallel copies of the reversed meridian $\mu^r$ and one copy of $\lambda$.  Furthermore, with an additional point $z'$  the diagram specifies the knot $\mu\subset Y_{-N}(K)$.  As above, the generators of $CFK^{\infty}(Y_{-N}(K),\mu)$ arising from this diagram are split according to whether they are supported in the  winding region  or are exterior.  The exterior generators  are characterized by the fact that the point of intersection occurring on $\lambda^{-N}$  lies outside the winding region  (recall that a generator is a $k$-tuple of intersection points between $\alpha$ and $\beta$ curves, with each $\alpha$ and $\beta$ curve appearing exactly once; thus $\lambda^{-N}$ is used exactly once by any $k$-tuple comprising a generator).  The exterior generators are in bijection with generators for the Heegaard diagram of $Y_0(K)$ (the diagram with $\lambda$ as the last curve).  Our bound of $2g(K)$ in the theorem will be attained if we can argue that the total number of $\SpinC$ equivalence classes represented by the exterior points is less than $|H_1(Y)|\cdot 2g$.   This follows  from the key properties of our Heegaard diagram.  Indeed, recall the first Chern class formula \cite[Proposition 7.5]{oz:hf-properties}:\begin{equation}\label{eq:chern} \langle c_1(\spinc_w({\bf x})), [\mathcal{P}]\rangle = e(\mathcal{P}) + 2\sum_{x_i\in {\bf x}} n_{x_i}(\mathcal{P}).\end{equation}
Here, $\bf x$ is a $k$-tuple generating a Heegaard Floer complex, $[\mathcal P]\in H_2$ is the second homology class obtained by capping off the boundary components of a periodic domain, $e(\mathcal{P})$ is the Euler measure of $\mathcal{P}$ (which agrees with the Euler characteristic for  periodic domains with all  multiplicities zero or one) and $n_{x_i}(\mathcal{P})$ is the average of the local multiplicities of $\mathcal{P}$ in the four regions surrounding an intersection point $x_i$.     For our particular Heegaard diagram for $Y_0(K)$, the right-hand side of \ref{eq:chern} becomes:
$$ -2g + 2\#\{x_i\in \text{interior}(\mathcal{P})\} + 2,$$ 
where $-2g$ is the Euler characteristic of $\mathcal{P}$. The additional $+2$ term comes from the fact that $\alpha_k$ and $\lambda$ do not intersect and must each contain an $x_i\in \bf x$.  Since $\alpha_k$ and $\lambda$ are on the boundary of $\mathcal{P}$,  each of these two $x_i$ have $n_{x_i}(\mathcal{P})=1/2$.  Finally, the fact that there are only  $2g$ other $\alpha$ curves which intersect $\mathcal{P}$ and that any $k$-tuple comprising a generator  must use one of these $\alpha$ curves for  the intersection point  $x_i\subset\lambda$   implies that 
$$ 0\le 2\#\{x_i\in \text{interior}(\mathcal{P}) \}\le 2(2g-1),$$ thus showing that 
$$ -2g+2\le \langle c_1(\spinc_w({\bf x})), [\mathcal{P}]\rangle \le 2g.$$
Now the fact that $\langle c_1(\spinc_w({\bf x})), [\mathcal{P}]\rangle$ is an even integer which vanishes if $c_1(\spinc_w({\bf x}))$ is torsion implies that there are at most $|H_1(Y)|\cdot 2g$ distinct $\SpinC$ equivalence classes represented on the Heegaard diagram for $Y_0(K)$, and hence the same bound exists for the number of exterior intersection points.  This completes the proof.
\end{proof}

%%%%%%%SECTION%%%%%%%%%%%%%%
\section{ The complex $CFK^\infty( S^3_{-N}(-2D_k), 2D_k)$   }\label{sectionhf2}

In general, the computation of the $d$--invariant of surgery on a knot $K\subset Y$ from $CFK^\infty(Y,K,\spinc)$ can be rather challenging; identifying patterns among the values that arise for various values of $\spinc$ is even more subtle.  If the surgery coefficient is appropriately large, however, there are significant simplifications.  This section describes the general theory and demonstrates that in our setting the simplifications that arise from the large surgery assumption do apply.   

To be more specific ~\cite[Theorem 4.1]{Knots} showed that the complex $CFK^\infty(S^3,K)$ determines  $CF^{+}(S^3_N(K), \spinc_m)$ for  $N \ge 2g(K) -1$, with a similar result proved for null-homologous knots in arbitrary $3$-manifolds. In~\cite[Theorem 4.1]{os4} this was generalized to rationally null-homologous knots, in which case   $CF^{+}(Y_N(K), \spinc_m)$ depends on the complexes $CFK^\infty(Y^3,K, \spinc'_{m'})$ for specified classes $\spinc'_{m'}$.  However, the generalization of \cite{os4} did not specify how large the framing parameter had to be in order to apply the result.  Rather, it simply showed that for sufficiently large framings such a formula exists, and then  a more general formula was proved which holds for arbitrary framings in terms of a mapping cone complex.   In our situation we will  apply a special case of the results of~\cite{os4}, taking advantage of the fact that $Y = S^3_{-2n}(-2D_{k})$, and that we are performing $2n$--surgery on a knot formed as the connected sum of a knot in $S^3$ with  the meridian of $-2D_{k}$.  While we utilize the full mapping cone complex, our surgery parameters are chosen so that they will be large enough for the simpler formula to hold.   This will manifest itself in a collapse of the mapping cone complex to a single term. In general, ``large" should be taken to mean: ``large in comparison to the Thurston norm of the complement."

Here is the statement of the result we need.  The exact correspondence between the \Sp\ structures $\spinc_m$ and $\spinc'_{m'}$ is implicit in the proof but is not needed in our application of the theorem.

 \begin{theorem}\label{dsurgthm}
Let $K_2 \subset Y =  S^3_{-N}(K_1)$ be a knot of the form $\mu \# K'_2$ where $\mu $ is the meridian of $K_1$ and $K'_2$ is a knot in $S^3$.  For any $N \ge \max (2g(K'_2)+2, 2g(K_1) )$, there is an enumeration of \Sp\  structures on $Y_{N}(K_2)$, $\{\spinc_m\}_{- N^2/2 \le m \le  N^2/2}$, such that $CF^+(Y_{N}(K_2),\spinc_m)$  is isomorphic to  $$CFK^\infty (S^3_{-N}(K_1),K_2,\spinc'_{m'})/ CFK^\infty(S^3_{-N}(K_1),K_2,\spinc'_{m'})_{\{i < 0 , j <m \}}   [\epsilon].$$
 The  elements in the above  quotient  with $i = 0, j \le m$ and $i \le 0, j= m$ are at filtration level 0 in $CF^+(Y_{N}(K_2) , \spinc_m)$; these represent     $\widehat{CF}(Y_{N}(K_2),\spinc_m)$.  The induced map  $U$ lowers filtration level by 1.
 The grading shift, $\epsilon$, is a function of $m$ and $N$, and   in particular, the grading shift does not depend on $K_2'$.
  \label{thm:gradingshift}
\end{theorem}

Applying this theorem to the relevant manifolds yields the following corollary:

\begin{corollary} For any $0\le k< n/2$, there is an enumeration of \Sp structures on  $M(K_{D_{k},n})$, $\{\spinc_m\}_{- 2n^2 \le m \le  2n^2}$ for which  $CF^{+} (M(K_{D_{k},n}),\spinc_m)$ is isomorphic to $$\frac{CFK^\infty (S^3_{-2n}(-2D_k),\mu\#2D_k,\spinc'_{m'})}{CFK^\infty(S^3_{-2n}(-2D_k),\mu\#2D_k,\spinc'_{m'})_{\{i < 0 , j <m \}}}   [\epsilon],$$
with filtration, grading shift, and $\F[U]$--module structure as in Theorem \ref{dsurgthm}.
  \end{corollary}
  \begin{proof}  $M(K_{D_{k},n})$ is obtained by $2n$--surgery on $\mu\#2D_k\subset S^3_{-2n}(-2D_{k})$.  Thus we need only verify that $2n\ge \max (2g(2D_k)+2, 2g(-2D_k))$, provided that $0\le k< n/2$.  Both $2D_k$ and $-2D_k$ have genus $2k$, being the connected sum of $2k$ copies of the Whitehead double, a genus one knot.  
  \end{proof}
%\end{corollary}

The rest of this section is devoted to proving Theorem~\ref{dsurgthm}.

%%%%SUBSECTION %%%%%%%%
\subsection{Heegaard diagrams,  \Sp\ structures, homology and surgery.}

Our computation of $HF^+(M,\spinc)$ relies on results of~\cite{os4},  in which the general problem of computing the Heegaard Floer homology of rational surgery on a knot in a rational homology sphere is studied.  Although the manifolds we consider are in some respects fairly simple, in order to apply~\cite{os4} it is essential to review some of the foundations. 

The manifold $M$ we are considering is formed by surgery on a link $(K',K) \subset S^3$ constructed from the Hopf link by placing local knots in each component. More specifically, $M$ is given by $-N$ surgery on $K'$ followed by $N$ surgery on $K$.   Thus, our  approach to computing the Heegaard Floer  homology of $M$ is to view it  as formed by performing $N$ surgery on  knot $K$, viewed as a knot in  $Y = S^3_{-N}(K')$.  We  begin by considering surgery on the Hopf link, in which case $Y = S^3_{-N}(U) = -L(N,1)$ and $M = L(N^2+1, N)$.  We then move to the more general case, encompassing the situation in which  the components are knotted.
 
%%%%SUBSECTION %%%%%%%% 
\subsection{Lens space Heegaard diagram.} \label{Heegaard} As a starting point, we consider lens spaces $-L(N,1)$.    On the left in Figure~\ref{figureheegaard} is a doubly pointed Heegaard diagram for $Y =- L(2,1)$, which we use to illustrate the general construction. 

\begin{figure}[ht]
\fig{.5}{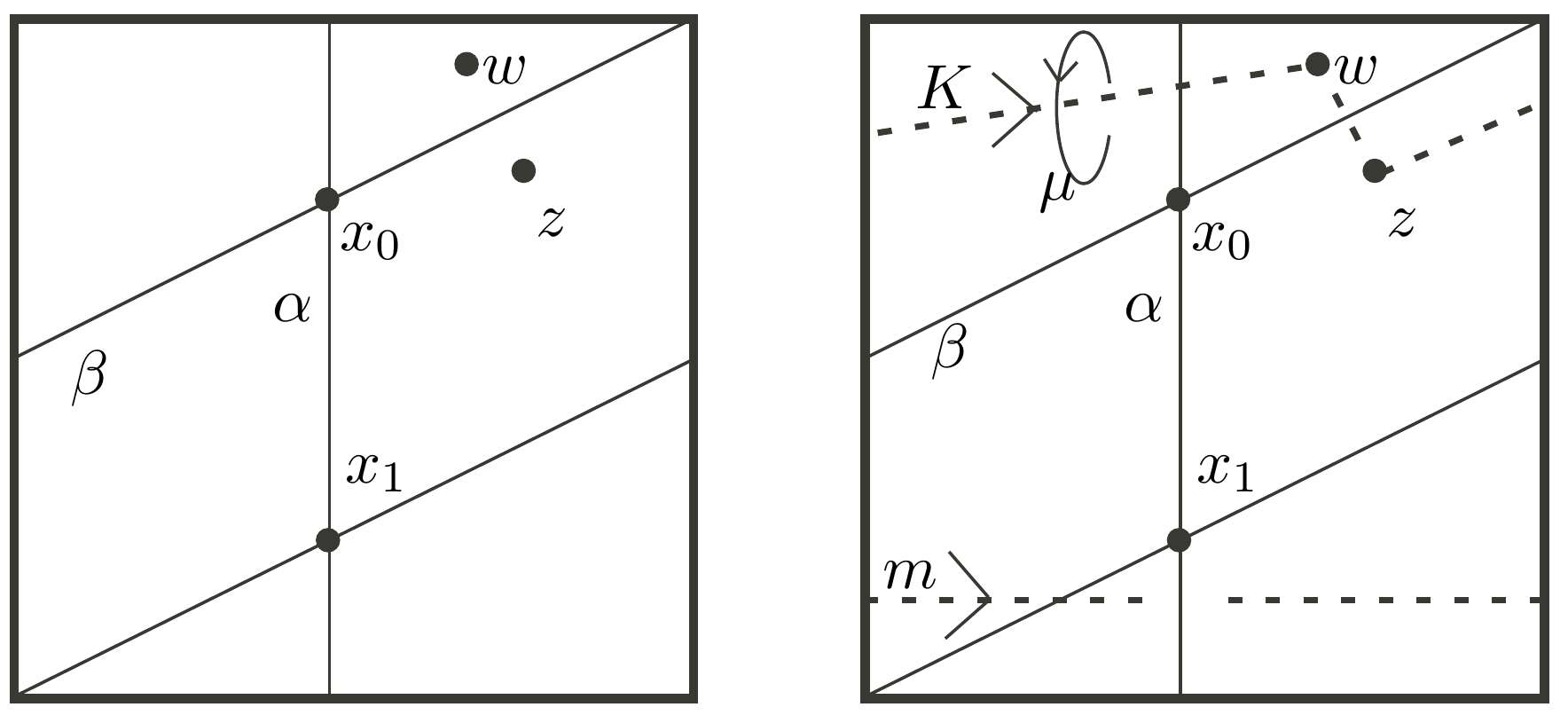}
\caption{Doubly pointed Heegaard diagram}\label{figureheegaard}
\end{figure}

In the lens space, the surface $\Sigma = T^2$  represented by this diagram bounds solid tori $U_\alpha$ and $U_\beta$ in which the curves $\alpha$ and $\beta$ bound embedded disks, respectively.  If we let  $\eta_\alpha$ be an arc from $w$  to $z$ on $\Sigma$ missing $\alpha$  that is pushed into $U_\alpha$ (except at its endpoints) and  let $\eta_\beta$ be an arc from $z$ to $w$ on $\Sigma$ missing $\beta$ pushed into $U_{\beta}$, the union of $\eta_\alpha$ and $\eta_{\beta}$ forms an oriented  knot $ {K}$ in $Y$.  Notice that  once isotoped into $U_\alpha$, $K$ represents the core of $U_{\alpha}$.

The meridian to ${K}$ we denote $\mu$.  The  complement of  ${K}$ in $U_\alpha$ is homeomorphic to  $T^2 \times I$ with $H_1(U_\alpha \setminus {K}) $ generated by $\mu$ and the curve $m$ illustrated on the right in Figure~\ref{figureheegaard}.  Notice that $H_1(Y \setminus  {K} )$ is generated by $\mu$ and $m$, subject to the relations $ Nm-\mu = 0$.  
 This is shown on the left in  Figure~\ref{figure2}, which illustrates the solid torus $U_\alpha$. Note that in the figure $K$ has not yet been isotoped into $U_\alpha$.

\begin{figure}[ht]
\fig{.45}{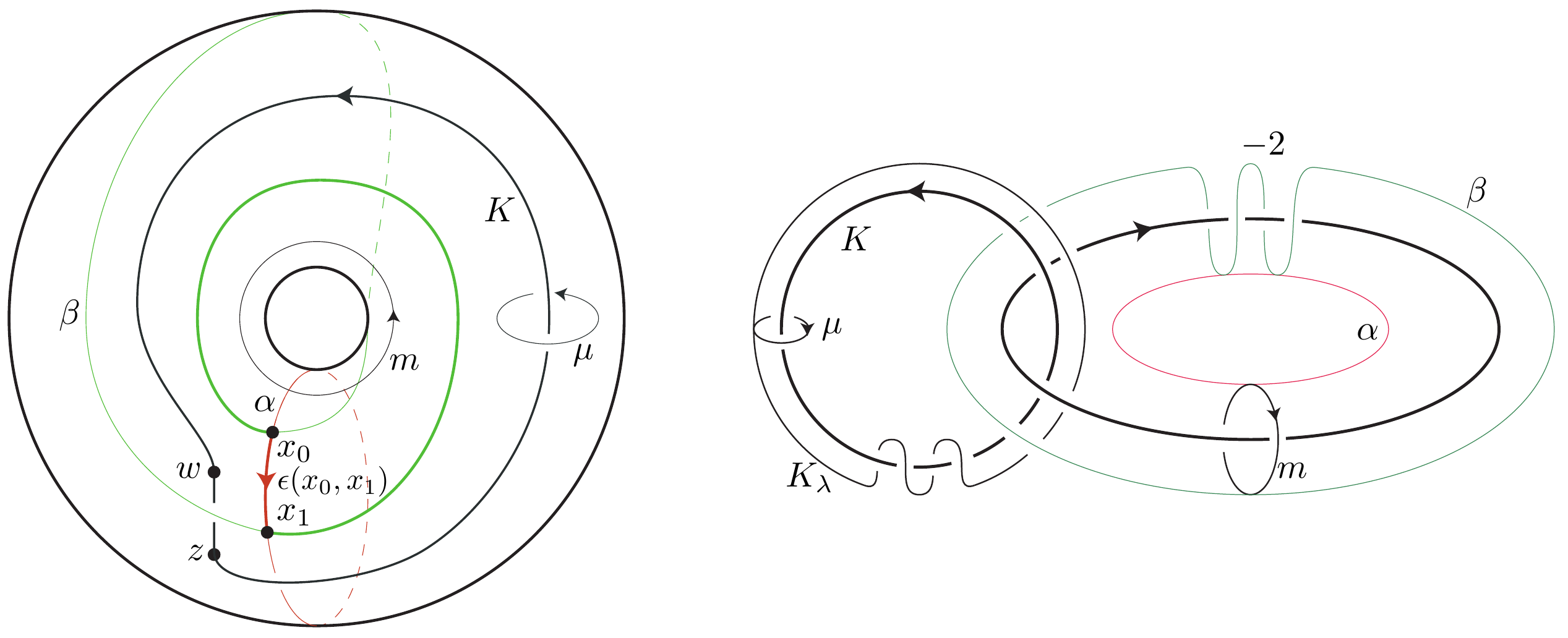} 
\caption{Surgery diagram of lens space $L(5,2)$}\label{figure2}
\end{figure}

%%%%SUBSECTION %%%%%%%%
\subsection{Relative Spin$^c$ structures.} Associated to each intersection point, $x_0$ or $x_1$ in the figures and   $\{x_0, x_1, \ldots, x_{N-1}\}$ for general $-L(N,1)$, there is a relative $\SpinC$ structure $
\spin_{w,z}(x_i) \in \SpinC(Y,{K}$).  The differences between these  satisfy \begin{equation}\label{spincdiff} 
\spin_{w,z}(x_{i+1})-
\spin_{w,z}(x_i) = \text{PD}[\epsilon(x_i,x_{i+1})]\in H^2(Y\setminus \nu K, \partial)\cong H^2(Y,K)\end{equation}
where $$\epsilon(x_i,x_{i+1}) \in \frac{H_1(\Sigma\setminus\{z,w\})}{\text{Span}\ \vec\alpha+\text{Span}\ \vec\beta}\cong H_1(Y\setminus \nu K)$$  is the class represented by a path that travels from $x_i$ to $x_{i+1}$ along $\alpha$ and then from $x_{i+1}$ to $x_i$ along $\beta.$ As seen from the figure, this curve is isotopic to $m$ in $Y \setminus\nu{K}$.  (In all these equations, $i \in \Z / N \Z$.)

%More generally, there is a relative $\SpinC$ structure  $\spin_{w,z}(\x) \in \SpinC(Y,{K})$  associated to a Heegaard Floer generator coming from a doubly-pointed Heegaard diagram adapted to a knot in an arbitrary $3$--manifold.    Equation (\ref{spincdiff})  holds in this context, with $\epsilon(\x,\y)\in H_1(Y\nu K)$ represented by a collection of paths along $\alpha$ and $\beta$ curves connecting the intersection points comprising $\x$ and $\y$.  See Section \cite[Section 2.2]{os4}   for more details.

 There is a natural map, called the {\it filling map}, $G_{Y,{K}}\!\co \SpinC(Y,{K}) \to \SpinC(Y)$ which satisfies $$G_{Y,{K}}(\xi +k ) -  G_{Y,{K}}(\xi   ) = \iota(k),$$ where $k \in H^2(Y,K)$. If ${K}^r$ denotes the orientation reverse of $ {K}$, then
$$G_{Y,{K}}(\xi   ) -  G_{Y,{K}^r}(\xi   ) = -\text{PD}[ {K} ].$$

\vskip.1in
\noindent{\bf Comment.}  As described by Turaev~\cite{turaev}, $\SpinC$ structures on a closed manifold correspond to equivalence classes of nonvanishing vector fields, where two are equivalent if homotopic off a ball.  In the case that $K \subset Y$ is an oriented  knot, a relative $\SpinC$   structure corresponds to a nonvanishing vector field on $Y \setminus \nbdk$ which points outwards on the boundary.  The map $G$ is given in terms of a canonical extension of a vector field from $Y\setminus\nbdk $ to $Y$.  See~\cite[Section 2.2]{os4} for a further discussion.
\vskip.1in

%%%%SUBSECTION %%%%%%%%
\subsection{$Y_N(K)$.}We are interested in performing $N$ surgery on $ {K}$.  To be clear about framings, in Figure~\ref{figure2} a push-off of ${K}$, $K_\lambda$, is illustrated.  The surgered manifold, $Y_N({K})$ is built by  removing a neighborhood of ${K}$ and replacing it with a solid torus so that $K_{\lambda}$ bounds a meridianal disk in that solid torus.     Note that $H_1(Y_N(K))$ is generated by $\mu $ and $m$ subject to the relations $Nm - \mu = 0$ and $m +N \mu = 0$.  For instance in the illustrated  case, with $N = 2$,   we get $H_1(Y_N(K)) = \Z / 5\Z$. (In fact, $Y_2(K) = L(5,2).$)   In general, for $N$ surgery on ${K}$ in $-L(N,1)$ we end up with $L(N^2+1, N) = -L(N^2 + 1, N)$.

 %%%%%%%SUBSECTION%%%%%%%%%%%%%%
 \subsection{The structure of $CFK^\infty(Y,K)$}
  
A relative $\SpinC$ structure $\xi\in \SpinC(Y,K)$ has  an associated doubly filtered chain complex  $CFK^\infty(Y,{K},\xi)$  generated by triples $[\x,i,j]$ satisfying 
\begin{equation}\label{triple}\spin_{w,z}(\x) + (i-j) \cdot \text{PD} [\mu] = \xi.\end{equation}
Here, $\x\in \Ta\cap\Tb$ is an intersection point of the Lagrangian tori in the symmetric product of a Heegaard diagram $(\Sigma,\alphas,\betas,z,w)$, and $i,j\in \Z$.
For instance, in the  case of $Y$ a lens space as above, we have illustrated examples in Figure~\ref{figurecomplexes}. In the figure, $x$ can denote any of the $x_i$ coming from the Heegaard diagram in Section \ref{Heegaard}.  The value of $\xi$ is written beneath each of the complexes.  (The shading in these diagrams becomes relevant later.)

 \begin{figure}[ht]
 \fig{.5}{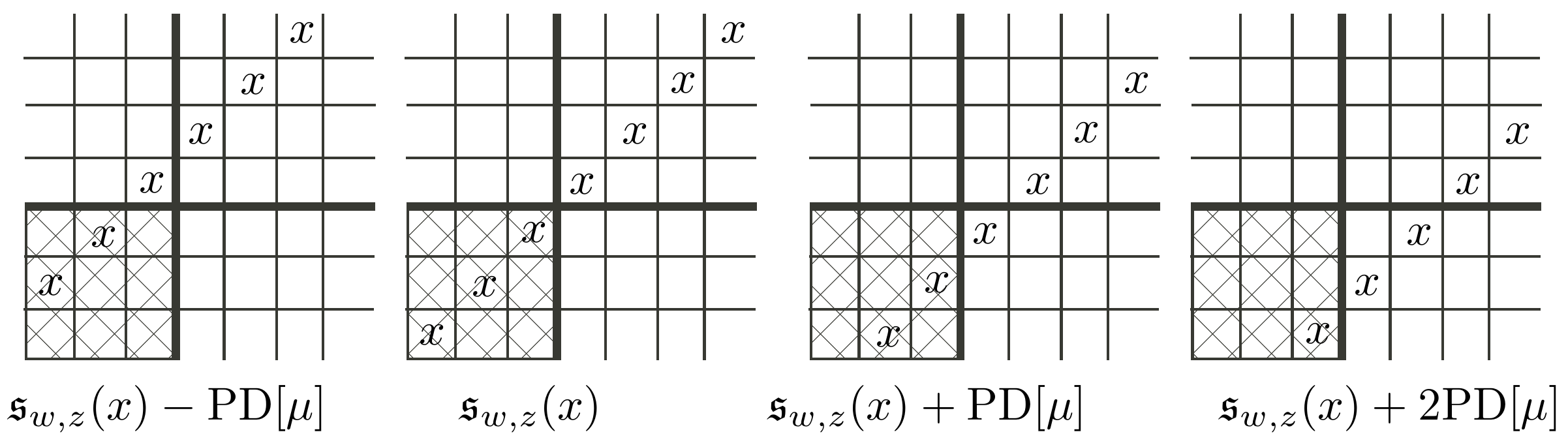} 
 \caption{$CFK^\infty(Y,K) $}\label{figurecomplexes}
\end{figure}

Every relative Spin$^c$ structure is  of the form $\spinc_{w,z}(x_0) + kPD[m]$ for  some $k \in \Z$, and this provides a correspondence between   Spin$^c(Y, K)$ and $\Z$.  Since $\mu = Nm$, the set of relative Spin$^c$ structures associated to each $x_i$ is a coset  of $N\Z \subset \Z$.   Also, since $\spinc_{w,z}(x_i) - \spinc_{w,z}(x_0)  = iPD[m]$, for different $x_i$ the sets are distinct cosets.

 %%%%%%%SUBSECTION%%%%%%%%%%%%%%
\subsection{Enumerating relative Spin$^c$ structures for the manifolds at hand}

We now move to our particular setting, in which $Y=S^3_{-N}(K')$ is a manifold constructed as surgery on a knot $K'\subset S^3$ and $K \subset Y$ is a knot of the form  $\mu\#J$, where $\mu$ is the meridian of $K'$  and  $J\subset S^3$.  Relative $\SpinC$ structures will play a central role in the surgery formula which  will be used to compute the Floer homology of $Y_N(K)$.  We discuss them now.

The manifold $Y_N(K)$ contains a knot, which we also denote $K$, induced by the surgery: $K$ is simply the core of the solid torus. There are two surjective filling maps to consider:
 $$G_{Y_N(K),{K}}  \co \SpinC(Y_N(K),{K} ) \to \SpinC(Y_N(K))$$   
 and
 $$G_{Y,K} \co \SpinC(Y,K) \to \SpinC(Y).$$
\noindent  There is a  canonical diffeomorphism  $$Y_N(K) \setminus\nbdk \to  Y \setminus\nbdk$$ that provides an identification between $\SpinC(Y_N(K),K)$ and $\SpinC(Y,K)$, with which we subsequently conflate elements in the two sets.  We will primarily think in terms of  $\SpinC(Y,K)$,  and the important point will be to understand the images of this $H^2(Y,K)$--torsor under $G_{Y_N(K),{K}}$ and $G_{Y,K}$.

Since $H^2(Y,K) \cong \zz$, we can (non-canonically) pick an affine isomorphism which enumerates the relative $\SpinC$ structures  on $Y\setminus \nbdk$ by integers (or elements in $H^2(Y,K)$).   For our purposes, it will be most convenient to pick an enumeration  that is compatible with the previously established affine isomorphism  $\SpinC(Y)\cong \Z/N\Z\cong H^2(Y)$  implicit in the statement of the refiltering theorem;  that is, we first enumerate elements in $\SpinC(S^3_{-N}(K'),\mu)$ to be compatible with our enumeration of $\SpinC$  structures on $S^3_{-N}(K')$, and we then use this to induce an enumeration  for $\SpinC(S^3_{-N}(K'),\mu\#J)$.  

To make this precise, recall that the refiltering theorem determines the $\Z\oplus\Z$--filtered homotopy type of  $\CF^\infty(S^3_{-N}(K'),\mu,\spinc_m)$, where $\spinc_m$ indicates a specific {\em absolute} $\SpinC$ structure on    $S^3_{-N}(K')=Y$ and $m\in \Z/N\Z$.  As in the last section, associated to a {\em relative} $\SpinC$ structure   $\xi\in \SpinC(Y,\mu)$, we obtain a complex $\CFK^\infty(Y,\mu,\xi)$ generated by triples satisfying (\ref{triple}).  We  pick an identification of $\SpinC(Y,\mu)$ with  $\Z$ so that the $m$-th relative $\SpinC$ structure, which we hereafter denote $\underline{\spinct}_m\in \SpinC(Y,\mu)$, or occasionally  by $m\in \Z$, has infinity complex given by the refiltering theorem; that is, 
$$ \CFK^\infty(Y ,\mu,\underline{\spinct}_m) = \CFK^\infty(Y,\mu,\spinc_m),  \ \ \mathrm{for} \ \lceil (-N+1)/2 \rceil \le m \le \lfloor N/2
  \rfloor,$$
where on the left we have the infinity complex associated to a relative $\SpinC$ structure and on the right the  filtered complex associated to the absolute $\SpinC$ structure labeled $\spinc_m$ by the refiltering theorem.  Equation (\ref{triple}) then determines the infinity complex for the remaining  relative $\SpinC$ structures (outside the interval of the theorem) by the equation:
$$ CFK^\infty(Y ,\mu,\underline{\spinct}_{m+kN})=CFK^\infty(Y ,\mu,\underline{\spinct}_m)\{0,-k\},$$
where $\{0,-k\}$ indicates that we have shifted the $j$-filtration down by $k$.   Finally, we observe that having picked an affine isomorphism  $\SpinC(Y,\mu)\cong H^2(Y,\mu)\cong\Z$, we subsequently obtain an affine isomorphism $\SpinC(Y,\mu\#J)\cong H^2(Y,\mu\#J)\cong\Z$, via the natural isomorphism $H^2(Y,\mu)\cong H^2(Y,\mu\#J)$.   % This labeling has the advantage of being the ``correct" labeling, in the sense that it agrees with a $\Z$-labeling of relative $\SpinC$ structures by Chern classes.  This latter labeling is defined by declaring $\underline{\spinct}_i$ to be the unique relative $\SpinC$ structure satisfying $$i = \frac{1}{2}\langle c_1(\underline{\spinct}_i)-PD([\mu]), [F,\partial F]\rangle.$$ Here, $[F,\partial F]\in H_2(Y \setminus K, \partial)\cong \Z$ is a generator, and $c_1(\underline{\spinct}_i)\in H^2(Y \setminus K, \partial)$ is the relative Chern class of the relative $\SpinC$ structure, defined with respect to a trivialization of the two-plane field dual to the defining vector field along the boundary (see \cite[Page 627]{linksOS} or \cite{NiLink} for more details).  

With our convention in hand, we  hereafter regard relative $\SpinC$ structures as integers, or as elements in $H^2(Y,K)$.  Similarly, we regard absolute $\SpinC$ structures on $Y$ or $Y_N(K)$ as elements in  $H^2(Y)\cong \Z/N\Z$ or $H^2(Y_N(K)\cong\Z/(N^2+1)\Z$, respectively.  Our convention is compatible with the filling maps, in the sense that they are now identified with the corresponding restriction maps on cohomology:
$$ H^2(Y_N(K),K)\to H^2(Y_N(K))$$
and
$$ H^2(Y,K)\to H^2(Y).$$

To illustrate these principles, and for use in the next section, let $\spinc$ be some fixed $\SpinC$ structure on $Y_N(K)$.  Now define $S(\spinc) = G_{Y_N(K),K}^{-1}(\spinc)$.  We have that $S(\spinc) = \{\underline{\spinct}_{k + (N^2+1)j}\}$ for $j \in \zz$ and some $k$, $0\le k \le N^2$.  Moreover, for each fixed value of $k$, there exists an $\spinc \in  \SpinC(Y_N(K))$ such that $\underline{\spinct}_k \in S(\spinc)$.

%%%%SUBSECTION %%%%%%%%
\subsection{The mapping cone} \label{mappingconesection}
$HF^+(Y_N(K), \spinc)$ can be computed as the homology of  a mapping cone complex built from $CFK^\infty(Y,K)$ via a construction  of \ons \ which we now recall.  We use the notation of  \cite{os3,os4} and refer the reader there for more details.

Letting $S=S(\spinc)$ be as above, there are complexes
$$\A_\spinc^+(Y,K) = \oplus_{{\xi} \in  S}   A_{{\xi}}^+(Y, K),$$
$$\B_\spin^+(Y,{K}) = \oplus_{ {\xi} \in S } B_{{\xi}}^+(Y, {K}).$$
Here  $$ A_{{\xi}}^+(Y, {K})= CFK^\infty(Y,K,{\xi})_{ \{ \text{max}(i,j) \ge 0\} },$$ and   
$$ B_{{\xi}}^+(Y, {K})= CF^+ (Y, G_{Y, K}({\xi})).$$ We can write    $$CF^+ (Y, G_{Y, K}({\xi})) = CFK^\infty(Y,K,{\xi})_{\{i \ge 0\}},$$ where in the term on the right of the equality, $K$ has provided a filtration of $B_{{\xi}}^+(Y,K)$. 

There are maps: 
$$v_\xi^+\co A_\xi^+(Y, {K}) \to B_\xi^+(Y, {K})$$
and 
$$h_\xi^+\co A_\xi^+(Y, {K}) \to B_{\xi+\text{PD}[K_\lambda]}^+(Y, {K}).$$  

The map $v$ is given by the projection map onto the quotient complex of $A_\xi^+(Y, {K})$ consisting of triples $[\x,i,j]$ with $i\ge 0$, the so-called vertical complex. 
 The map $h$ is  more subtle.   
Interchanging the roles of $i$ and $j$ replaces $K$ with $K^r$, its reverse.  The associated filling map for $K^r$ is denoted $G_{Y,K^r}$.  Because of the string reversal, $G_{Y,K^r}({\xi}) = G_{Y,K}({\xi}) +PD(K)$.   Thus, if we simply take the quotient corresponding to the horizontal projection, the target of this chain map is a complex homotopy equivalent to $\CF^+(Y, G_{Y,K}({\xi}) +PD(K))$. The map  $h_{\xi}^+$ is given by horizontal projection, followed by this chain homotopy equivalence. 

We now want to consider the set $S(\spinc)$ in terms of  $\SpinC$ structures on $Y$. To do so we write $$  S(\spinc) = \{ \underline{\spinct}_{k + (N^2+1)j} \}_{j\in \zz}$$ for some fixed $k$ satisfying $0 \le k \le N^2$. This set can be partitioned according to its $N$ possible images in $\SpinC(Y)$ under the filling map $G_{Y,K}$.  Let $0 \le l \le N-1$.  Then $S$ can be written as 
$$\bigcup_{0\le l \le N-1}\left( \{\underline{\spinct}_{ [j(N^2 +1) +(l-k)N]N   + l                } \}_{j \in \zz} \right).$$ Recalling that $\mu = Nm$, this can be rewritten as
$$\bigcup_{0\le l \le N-1} \left(  \{\underline{\spinct}_{  l  } +  [j(N^2 +1) +(l-k)N]PD(\mu)   \}_{j \in \zz}    \right).$$
Deriving the following formula is rather delicate, but its validity is easily checked:  
$$l + [j(N^2 +1) +(l-k)N]N = l \mod N$$ and$$l + [j(N^2 +1) +(l-k)N]N = k \mod N^2+1.$$

 %%%%%%%SUBSECTION%%%%%%%%%%%%%%
\subsection{Reduction to a finite complex}
From this discussion it is apparent that, in general, the mapping cone complex is fairly complicated.  In this subsection we observe that it always reduces to a complex that is a  quotient of a finite dimensional complex over $\ff[U, U^{-1}]$.  In the next subsection we observe that in our special case  the complex reduces to a single $A_{{\xi}}$ term.

Consider the complexes ${\A} = \oplus A_i$ and ${\B} = \oplus B_i$, joined by the chain map $D$ as illustrated below.  We denote the mapping cone complex of $D$ by $\C$.  Since $CFK^\infty$ is finitely generated over $\ff[U,U^{-1}]$, it follows that $v\co A_i \to B_i$ is an isomorphism for all large $i$, and $h\co A_i \to B_{i+1}$ is an isomorphism as $i$ goes to negative infinity.  The diagram below presents a special case. 
\vskip.2in
\begin{center}
$\cdots
\begin{diagram}
\node{A_{-3}} \arrow{se,l}{\hskip-.07in\cong} \node{A_{-2}} \arrow{s}\arrow{s,r}{v}  \arrow{se,l}{\hskip-.07in\cong} \node{ {  {A_{-1}}}    } \arrow{s,r}{v}    \arrow{se,l}{h} \node{A_{0}  }   \arrow{s,r}{v}  \arrow{se,l}{h} \node{   A_{1} } \arrow{s,r}{\hskip-.02in\cong} \arrow{se,l}{h} \node{A_{2}} \arrow{s,r}{\hskip-.02in\cong} \arrow{se,l}{h} \node{}  \\
\node{} \node{B_{-2}}   \node{ {  {B_{-1}}}    }   \node{B_{0}  }   \node{   B_{1} }    \node{B_{2}} \node{B_{3}}
\end{diagram}\cdots
$
\end{center}

\noindent In this example, we have the following subcomplex, ${\C}'={\A}' \oplus {\B}'$:\vskip.2in

\begin{center}
$\cdots 
\begin{diagram}
\node{A_{-3}} \arrow{se,l}{\hskip-.07in\cong} \node{A_{-2}} \arrow{s}\arrow{s}  \arrow{se,l}{\hskip-.07in\cong} \node{}  \node{   } \node{   A_{1} } \arrow{s,r}{\hskip-.02in\cong}\arrow{se} \node{A_{2}} \arrow{s,r}{\hskip-.02in\cong}\arrow{se}  \node{} \\
\node{} \node{B_{-2}}   \node{ {  {B_{-1}}}    }   \node{   }   \node{   B_{1} }    \node{B_{2}} \node{B_{3}}
\end{diagram} \cdots
$
\end{center}

The restriction of $D$ to this subcomplex, which we denote $D'$,  induces an isomorphism $D'_*\co H_*(\A')\to H_*(\B')$.  Injectivity is evident; surjectivity follows from the fact  that for each $x$ in  the right portion of the complex, $(h\circ v^{-1})^k (x) = 0$  for some   $k$. Similarly, for each $x$ in the left portion of the complex, $(v\circ h^{-1})^k(x) = 0 $ for some   $k$. There is a long exact sequence $$\to H_*({\B'}) \to H_*({\C'}) \to H_*({\A'})\to $$ with connecting homomorphism given by $D'$. Thus, $H_*({\bf \C'}) = 0$.

Consider next the short exact sequence $0\to {\C'} \to {\C} \to {\C/\C'}\to 0$; it leads to a long exact sequence, and we see that $H_*({\C}/{\C'}) = H_*({\C})$.  That is, the homology of ${\C}$ is the homology of the complex 

\begin{center}
$ 
\begin{diagram}
\node{} \node{} \node{ {  {A_{-1}}}    }   \arrow{se}  \node{A_{0}  }  \arrow{s}  \node{   }   \node{ } \node{}\\
\node{} \node{ }   \node{    }   \node{B_{0}  }   \node{ }    \node{} \node{}
\end{diagram}
$
\end{center}

Notice that had $h\co A_{-1} \to B_0$ also been an isomorphism in this example, then the complex would have reduced to a single term, $A_0$.  This occurs in the cases of lens spaces that arise in our work, $L(N^2+1, N)$.  We will see in the next section that this total collapse also occurs for our manifolds $M$.

%%%%%%%SUBSECTION%%%%%%%%%%%%%%
\subsection{General complete collapse of the mapping cone complex}  In the case of lens spaces constructed as surgery on the unknot, the $\CFK^\infty$ complexes which arise are all of the form $(C \otimes_\ff \F[U,U^{-1}])\{0,k_i\}$, where $C$ is a 1--dimensional doubly filtered $\ff$ vector space  generated by a single vector $x_i$ at filtration level $(0,0)$.  The shift $\{0,k_i\}$ is a   $j$--filtering shift of $k_i$ for appropriate integers $k_i$.   It thus follows quickly that there is an $a$ such that $v_i$ is an isomorphism for all $i \ge a$ and $h_i$ is an isomorphism for all $i \le a-1$.  This explains our comment above that for lens spaces there is a complete collapse of the $(\A,\B)$ mapping cone complex to a single $A_i$.

In the more general situation that appears for our $M$, the $\CFK^\infty$  complexes which arise are of the form $(C_{\bar{\imath}} \otimes_\ff \F[U,U^{-1}])\{0,k_i\}$ for finite dimensional doubly filtered $\F$-chain complexes $C_{\bar{\imath}}$ which are no longer 1--dimensional (here $\bar{\imath}=i\mod n$ for some $n$). In particular, the $\CFK^\infty$ complexes are not restricted to a single diagonal.  Instead, they lie in a band; in Figure~\ref{bandfig} we illustrate a case in which the band is of height six.  

Notice that in the example illustrated in Figure~\ref{bandfig}, the vertical quotient is not an isomorphism, but the horizontal quotient is.  In general, one of the two maps will be an isomorphism unless the origin is contained in the band.  Furthermore, if this band is shifted up (by $-2$ or more) then $h$ continues to be an isomorphism, and if it is shifted down by seven or more, the vertical map becomes an isomorphism.  
 
\begin{figure}[h]
\fig{.5}{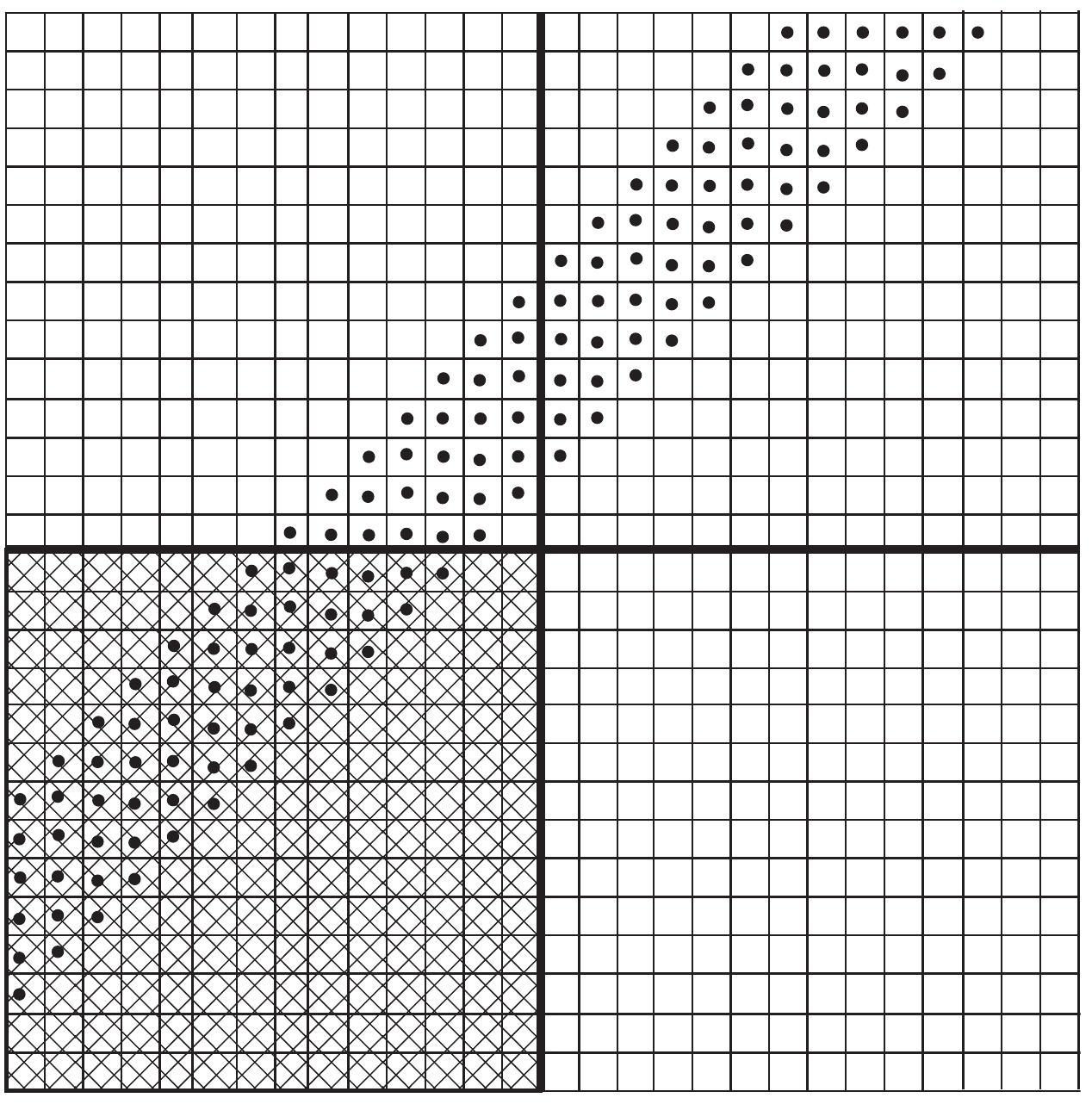} 
\caption{}\label{bandfig}
\end{figure}

Recall now that in our decomposition $\A^+_{{\spinc}}(Y,K) = \oplus_{  {\xi} \in S}\A^+_{{\xi}}(Y,K)$ we have 
$$S=\bigcup_{0\le l \le N-1}\left( \{\underline{\spinct}_{  l  } +  [j(N^2 +1) +(l-k)N]PD(\mu)   \}_{j \in \zz} \right).$$

In order to state the next result, let the  width $\Wd(C)$ of a doubly filtered complex be defined as: $\Wd(C) = \max (i-j) -\min(i-j) +1$, where the minimum and maximum are taken over all pairs $(i,j)$ such that there is a nontrivial filtered generator of filtered degree $(i,j)$.  Roughly,  $\Wd(C)$ represents the width of the narrowest $U$--invariant band which contains the full complex.  The width determines the Thurston norm of the knot complement \cite{NiLink, GenusBounds}. 

\begin{theorem}[\bf Collapse Theorem] \label{collapsethm}Suppose that for each $\underline{\spinct}_l \in \SpinC(Y,K)$ with $0 \le l \le N-1$, the complex $C = CFK^\infty(Y,K, \underline{\spinct}_l )$ satisfies $\Wd(C) \le N$.  Then the mapping cone complex $\A \to \B$  that determines $HF^+(Y_N(K),\spinc)$ collapses  to a single $A_i$ for some $i$.
\end{theorem}
 
 \begin{proof} Recall that as in Section~\ref{mappingconesection}, $k$ is a specified fixed integer, $0 \le k \le N^2$.   For simplicity, denote $CFK^\infty(Y,K, \underline{\spinct}_l )\{0,s\}_{ \{ \text{max}(i,j) \ge 0\} }$  by $A'_l(s)$ for $0 \le l \le N-1$ where, as above, $\{0,k\}$ indicates that we have shifted the doubly filtered complex up by $k$.  Then the $A_i$ that occur are ordered as follows if we begin with $l=0$ and $j = 0$:
  $$\ldots A'_{N-1}(N+1+kN)  , A'_{0}(kN) ,  A'_{1}((k-1)N), \ldots ,\hskip.3in$$ $$\hskip.75in\ldots A'_{N-1}((k+1-N)N), A'_{0}(kN-1-N^2), \ldots$$
  
 Notice that the shifts increase by $N$, or when going from $A'_{N-1}$ to $A'_{0}$, by $N+1$. It follows that at most one of $A_i$  is in a band which includes the origin, with all greater $A_i$ being in bands below the origin and all lesser $A_i$ being in bands above the origin. Thus the complex collapses to a single $A_i$, as desired.  
 \end{proof}

%\begin{corollary} The mapping cone complex $\A \to \B$ used to compute the complex $CF^{+} (M(K_{D_{k},n}))$ collapses to a single $A_i$ factor if $k \le n/2$.  
%\end{corollary}

%\begin{proof}
%The manifold $M(K_{D_{k},n})$ is constructed by   performing surgery on a knot whose $CFK^\infty$  complex is of the form $CFK^\infty(S^3_{-2n}( {-2D_{k}}), \mu, \spinc) \otimes^{2k}CFK^\infty(S^3, D)$  (where $D$ is the Whitehead double of the right-hand trefoil, as before, and $\mu$ is the meridian to the connected sum of $2k$ copies of its mirror image, viewed as a knot inside the manifold obtained by $-2n$ surgery).
%By the refiltering theorem, Theorem~\ref{thmrefilter}, the complex $CFK^\infty(S^3_{-2n}( -2D_{k}), \mu, \spinc)$ has width at most two for any $\spinc\in \SpinC(S^3_{-2n}( -2D_{k}))$.  The complex $CFK^\infty(S^3, D)$ has width three, as shown in Appendix~\ref{appendixcfkd}.  A simple exercise gives the addition formula, $\Wd(C_1 \otimes C_2) = \Wd(C_1) + \Wd(C_2) -1$.  It follows that the complex $C$ satisfies $\Wd(C) = 2k+2$.  Combining these results with Theorem~\ref{collapsethm}, we see that the complex collapses if $2k+2 \le 2n$.  
%\end{proof}

We now have all the pieces necessary to prove Theorem \ref{dsurgthm}: 

\begin{proof}[{\bf Proof  of Theorem~\ref{dsurgthm}.}]   Given that $N$ is greater than $2g(K_1)$, we can apply the refiltering theorem (Theorem~\ref{thmrefilter}) to show that the complex  $CFK^\infty(S^3_{-N}( K_1), \mu, \spinc)$ has width at most two for any $\spinc\in \SpinC(S^3_{-N}(K_1))$.   The K{\"u}nneth theorem  for the knot Floer homology of connected sums (\cite[Theorem 7.1]{Knots} c.f. \cite[Theorem 5.1]{os4}) implies $$CFK^\infty(S^3_{-N}( K_1), \mu\#K'_2, \spinc)\simeq \CFK^\infty(S^3_{-N}(K_1),\mu,\spinc)\otimes \CFK^\infty(S^3,K'_2),$$
for any knot $K'_2\subset S^3$ and any $\spinc\in \SpinC(S^3_{-N}(K_1))$.    Now \cite[Theorem 1.2]{GenusBounds} implies that the width of $\CFK^\infty(S^3,K'_2)$ is equal to $2g(K'_2)+1$, and a simple exercise gives the addition formula, $\Wd(C_1 \otimes C_2) = \Wd(C_1) + \Wd(C_2) -1$.   Thus the width of the complex for $K_2=\mu\#K'_2\subset S^3_{-N}(K_1)=Y$ is at most $2g(K'_2)+2$.

Thus, according to the collapse theorem (Theorem  \ref{collapsethm}), for  each \Sp\ structure $\spinc$ on $Y_N(K_2)$, the homology $CF^+(Y_N(K_2), \spinc)$ is given by a single complex $A_i$.  This complex is of the form $C^\infty / C^\infty_{\{i<0, j<0\}}$ where $C^\infty$ is the complex $CFK^\infty(Y, K_2, \spinc')$ shifted down by some parameter $m$,  $-N^2/2 \le m \le N^2/2$, and where $\spinc'$ is some \Sp\  structure on $Y$.  Alternatively, it is   the quotient $$ CFK^\infty (Y, K_2, \spinc')/ CFK^\infty(Y,K_2, \spinc')_{\{i<0, j<m\}}.$$  The gradings are shifted, but the shift is independent of the choice of $K_1$ and $K'_2$.   

The action of $U$ is to shift downward along the diagonal.  Thus, the kernel of the $U$ action is precisely the set of elements at filtration level 0 as described in the statement of Theorem~\ref{dsurgthm}.
\end{proof}
\vskip.2in

%%%%%%%SECTION%%%%%%%%%%%%%%
\section{Computations}\label{computesection}

%%%%SUBSECTION %%%%%%%%
\subsection{Knot complexes}  

For a given $n$  and $k$ we have defined $K_{D_{k},n}$ to be the knot shown in Figure \ref{figureknot}, with the knot $J$ given by $k D$ (where $D$ continues to denote the positive-clasped untwisted Whitehead double of the right-handed trefoil).  In this case, $M(K_{D_{k},n}) $ is given as  $(-2n, 2n)$--surgery on the link formed from the Hopf link by replacing the first component with $-2kD$ and the second component by $2kD$; see Figure~\ref{figurecover}.  As mentioned earlier, $n$ will be in the set $\caln$ described in Proposition~\ref{defineNprop} and Proposition~\ref{thm:number}.  For each $n$ we will choose a value for $k$, denoted $k_n$, selected to satisfy certain properties.    
A key result, which follows from the work in Appendices~\ref{appendixcfkd} and~\ref{appendixt22k+1}, is the following.

\begin{prop} \label{prop:DandT23}$\ $

\begin{itemize}
\item  The chain complex $\cfk^\infty(S^3,D)$ is filtered chain homotopy equivalent to
the chain complex $\cfk^\infty(S^3,T_{2,3})\oplus A$, where $A$ is an acyclic
complex. If $[x,i,j] $ is a filtered generator of $\cfk^\infty(S^3,D)$, then $|i-j|\le 1$.

\item The chain complex $\cfk^\infty(S^3,D_k)$ is filtered chain homotopy equivalent to
the chain complex $\cfk^\infty(S^3,T_{2,2k+1})\oplus A$, where $A$ is an acyclic
complex. If $[x,i,j] $ is a filtered generator of $\cfk^\infty(S^3,D_k)$, then $|i-j|\le k$.

\end{itemize}
\end{prop}

\begin{proof} The first statement expands on the computation of $\widehat{CFK}(S^3, D)$ given in~\cite{hedden1}.  Its proof occupies Appendix~\ref{appendixcfkd}.  The second statement follows from the relationship between $CFK^\infty(S^3, T_{2,3})^{\otimes k}$ and 
 $CFK^\infty(S^3, T_{2,2k+1}) $ described in Theorem~\ref{appendixt22k+1thm}.
\end{proof}

We next compute the knot Floer complex of the meridian of the connected sum of $2k$ copies of the mirror of the doubled trefoil, in the space formed by surgery upon this connected sum. 
\begin{theorem}\label{theoremx}
For $2n \ge 4k$ and $-n+1 \le m \le n$, the doubly filtered  complex $CFK^\infty(S^3_{-2n}(-2D_{k}),\mu, \spinc_m)$ is chain homotopy equivalent to the complex $C_{2n, k, m} \cong (T \oplus A) \otimes_\ff \ff[U,U^{-1}]$ where $A$ is a finitely generated acyclic complex and $T$ has one generator at filtration level $(0, 0)$ or $(0,-1)$. More precisely, the generator of $T$ has filtration level $(0,0)$ if $m< -2k$ or $m$ odd $< 2k$, and has filtration level $(0,-1)$ if $m\ge 2k$ or $m$ even $\ge -2k$. For any filtered generator $[x,i,j]$, $|i-j| \le 1$.
\label{thm:cfk of mu}
\end{theorem}

\begin{proof}The theorem will be a direct application of the refiltering theorem (Theorem~\ref{thmrefilter}) together with the previous proposition.  To begin, note that  since the genus of $-2D_k$ is $2k$, we can use the refiltering theorem provided that $2n\ge 4k$ (as assumed). Applying  the tensor product to the formula given in Proposition~\ref{prop:DandT23}, we prove in the appendix (Theorem~\ref{appendixt22k+1thm})   that there is a ($\Z\oplus\Z$-filtered) chain homotopy equivalence $$CFK^\infty(S^3,  2k D)\simeq CFK^\infty(S^3,T_{2,4k+1})\oplus A$$ where $A$ is an  acyclic complex.  Recalling that $CFK^\infty(S^3,-K)=CFK^\infty(S^3,K)^*$, we obtain a corresponding decomposition for the mirrors:
$$CFK^\infty(S^3,  -2k D)\simeq CFK^\infty(S^3,-T_{2,4k+1})\oplus A^*.$$
   Applying the refiltering theorem then   gives a decomposition $$CFK^\infty(S^3_{-2n}(-2D_{k}),\mu, \spinc_m)\simeq \CFK^\infty(S^3_{-2n}(-T_{2,4k+1}),\mu,\spinc_m)\oplus A',$$ where $A'$ is an acyclic complex concentrated on one or both of the diagonals mentioned in the theorem (note that, by an abuse of notation, $\mu$ is the meridian to $-2D_k$ and $-T_{2,4k+1}$ on the left and right sides of the equivalence, respectively).  Precisely, $A'$ is the $\Z\oplus\Z$-filtered chain complex which  results from the  refiltration of $A^*$.
   
   Thus it remains to understand  the result of applying the refiltering theorem   to the complex $\CFK^\infty(S^3,-T_{2,4k+1})$.   For reference, the figure shows the complexes for $T_{2,5},T_{2,9}$, and $-T_{2,9}$.  Applying Theorem~\ref{thmrefilter}, one  sees that for each  \Sp\ structure $\spinc_m$, the complex $CFK^\infty(S^3_{-2n}(-T_{2,4k+1}),\mu, \spinc_m)$ is given by a complex concentrated on the diagonal and one below the diagonal.   We wish to understand this complex better.

 \begin{figure}[h]
\fig{.5}{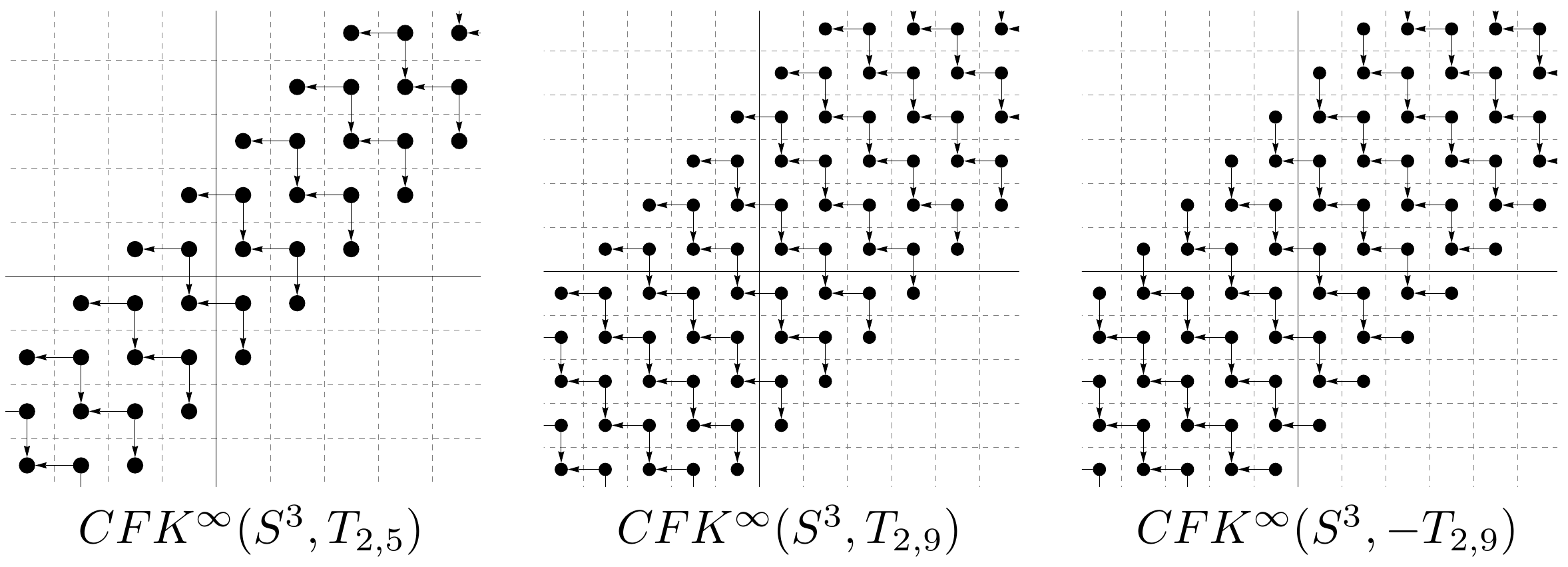}
\caption{}\label{figure2-4r}
\end{figure}

For example,  Figure~\ref{figure:t29}(a)  illustrates the complex  $CFK(S^3_{-2n}(-T_{2,9}),\mu, \spinc_{-3})$, in which we have labeled two of the generators $x$ and $y$.  Replacing $y$ with $x+y$  gives a filtered change of basis, and the new complex is as shown in Figure~\ref{figure:t29}(b).  Notice that this has introduced an acyclic piece.  Repeating the process yields the complex illustrated in Figure~\ref{figure:t29}(c). Applying this simplification in general  shows that for each $m$, the complex $CFK^\infty(S^3_{-2n}(-T_{2,4k+1}),\mu, \spinc_m)$ splits as a direct sum of an acyclic complex (necessarily on  the two stated diagonals) plus a complex of the form $T\otimes \F[U,U^{-1}]$, where $T$ is a single generator of the stated filtration (in fact the complex is filtered homotopy equivalent to $T\otimes \F[U,U^{-1}]$).  This completes the proof of  Theorem~\ref{theoremx}.

\begin{figure}[h]
\fig{1}{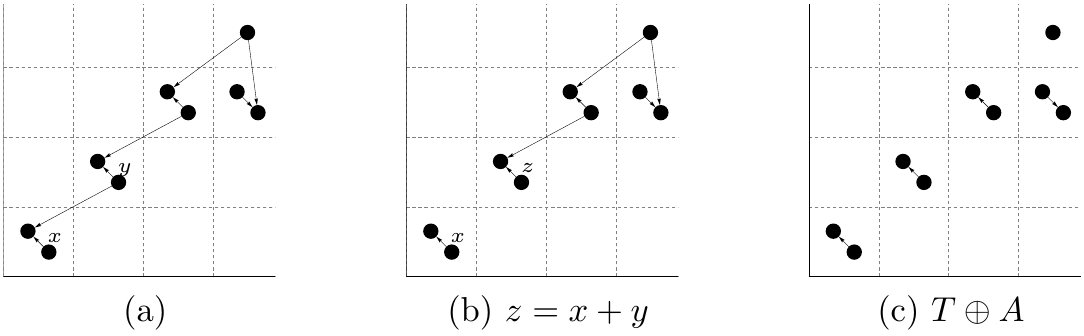} 
\caption{}\label{figure:t29}
\end{figure}

 \end{proof}

We next want to consider the second component of the link.  This is obtained from the meridian of the first component by forming the connected sum with a knot whose  $CFK^\infty$ is  identical, modulo acyclic summands, to  $T_{2, 4k+1}$. Moreover, the complex for the meridian, modulo acyclic summands, is simply that of the unknot with a filtration shift.  Given these observations, the following result is immediate.  
\begin{theorem} \label{thmt24k+1}
For $2n \ge 4k$ and for $-n+1  \le m \le  n$, we have  $$CFK^\infty(S^3_{-2n}(-2D_{k}), \mu\#2D_{k}, \spinc_m)_{i,j} = CFK^\infty(S^3, T_{2,4k+1})_{i, j - \delta} \oplus A,$$
with $A$ an acyclic complex.  Here, $\delta=0$ if $m< -2k$ or $m$ odd $< 2k$;  $\delta=-1$ if $m\ge 2k$ or $m$ even $\ge -2k$. Cycles representing nontrivial classes of grading 0 are located at filtration levels $i+j=2k+\delta$.  For any filtered generator $[x,i,j]$, $|i-j| \le 2k +1$.
\end{theorem}

We will need to compare this with the case of $J$ the unknot, for which  the computation is simpler.  The result is as follows.

\begin{theorem} \label{thmsuu}
For $2n \ge 3$ and for $ -n+1\le m \le n$, we have  $$CFK^\infty(S^3_{-2n}(-U), \mu\#U, \spinc_m)_{i,j} = CFK^\infty(S^3, U)_{i, j - \delta}.$$
Here, $\delta=0$ if $m<0$ and $\delta=-1$ if $m\ge 0$.
The cycle representing a nontrivial homology class is at filtration level $(0,\delta)$.
\end{theorem}

%%%%SUBSECTION %%%%%%%%

\subsection{$d$--invariants and acyclic summands}  As already seen, many of the  complexes that arise have  included acyclic summands.  We will need to see that these summands do not affect the computations of  the relevant $d$--invariants.   Rather than present the  most general theorem concerning acyclic summands, we will restrict ourselves to a simpler setting  for which the proof is more straightforward. 

 Let $\cald$ be a free, finitely generated $\ff[U,U^{-1}]$--chain complex that is $\Q$--graded.  Moreover, suppose that $\cald$ has a distinguished basis, and is $\Z$--filtered (by subcomplexes) by the corresponding distinguished    $\F[U]$-submodules   $$ ...\subset U^{k}\F[U]\subset U^{k-1}\F[U]\subset U^{k-2}\F[U]\subset ...$$  Thus the action of $U$ lowers filtration level by one.  Assume that it lowers  grading by two.  We let $d(\cald)$ denote the least grading of a nontrivial homology class $z \in H(\cald / \cald_{i<0} )$ where $z$ is in the image of $U^k$ for all $k$; here $U^k$ is  viewed as an endomorphism of $ H(\cald / \cald_{i<0} )$.  (If such an element does not exist, then $d(\cald) = -\infty$.)
 
 A particular example can be built from a  finitely generated acyclic $\ff$--chain complex  $A$ which is filtered and graded: regard a filtered generator $\x$ of $A$ as a monomial $\x\otimes U^{-(\mathrm{Filtration\ of\ } \x)}$ and form the $\ff[U,U^{-1}]$--complex $\cala =  A \otimes \ff[U, U^{-1}]$, so that $A \otimes 1$ has the same filtration and grading as $A$,  and $U$ acts on the right, decreasing filtration level by   one and grading by two.  Write $A_k = A \otimes U^k$. Thus an element in $A_k$ has filtration level equal to the filtration level of the corresponding element in $A$, shifted down by $k$.

\begin{prop}\label{acyclicprop} If $\cald' \cong \cald \oplus \cala$ with $\cald$ and $\cala$ as above, then     $d(\cald') = d(\cald)$.

\end{prop} 
\begin{proof}  Forming the quotient complex of $\cala$ with the subcomplex $\cala^- = \cala_{i<0}$ of elements with  filtration level less than 0 yields a complex $\cala^+ = \cala / \cala^-  $.  This complex decomposes over $\ff$ as  $\oplus_k  A_k  / (A_k \cap \cala^-)$.

Since $A$ is finitely generated, there is an $N$ such that: (1)  if $k >N$, then $(A_k \cap \cala^-) = A_k$; and (2) if $k < -N$, then $(A_k \cap A^- = 0)$.  Recalling that $A$ is acyclic, we see that for all $k$ with $|k| >N$, the homology group $H(  A_k  /  (A_k \cap \cala^-) ) = 0$. 

The action of $U$ maps $H(   A_k / (A_k \cap \cala^-) ) $ to $ H(  A_{k+1}  / (A_{k+1} \cap \cala^-) ) $.  The only possible nontrivial elements in the homology of $\cala / \cala^-$ are sums of elements in $ H(  A_{k+1}  / (A_{k+1} \cap \cala^-) ) $ for $|k|  \le N$.  But no such element can be in the image of $U^{2N}$ since it would then be in the image of an element in $H(A_{k} / (A_k \cap \cala^-) )$, for some $k<-N$, and we have seen these groups are trivial.

Given this, we see each nontrivial elements of $H(\cald' / \cald'_{<0})$ that is in the image of $U^k$ for arbitrarily large $k$  is also in the  image of an element of $H(\cald / \cald_{<0})$.  

\end{proof}

%%%%SUBSECTION %%%%%%%%
\subsection{Computations of $d$--invariants}

We   need to  compute the difference of $d$--invariants, $d(M(K_{D_{k},n}),\spinc_n) -d(M(K_{U,n}),\spinc_n)$, for any $k$ with $0\le k<n/2$.
Recall that $M(K_{D_{k},n})$ is also denoted by $S^3_{-2n,2n}(-2D_{k},2D_{k})$.        

In the following theorem we use $\epsilon_i$ to denote a grading shift.  As stated in the theorem, these are homological invariants that depend on the value of $n$, but are independent of the particular knots chosen.  Thus, the values of $\epsilon$ in the first two equations and in the last two equations are equal and their particular values irrelevant.  For this reason we denote them  simply by  $\epsilon_1$ and $\epsilon_2$.  We also include in the statement the number $d(S^3, \spinc_0)$ despite it equalling 0; this highlights the  role of the  $d$--invariant of the base space in which the knot lies.

\begin{theorem}\label{diff_of_d_thm}  For any $0\le k<n/2$
\begin{align}  
 0&= d(S^3_{-2n}(-2D_{k}),\spinc_{-n})-d(S^3,\spinc_0)-\epsilon_1 \label{deq1} \\
 0&= d(S^3_{-2n}(-U),\spinc_{-n} )-d(S^3,\spinc_0)-\epsilon_1 \label{deq2} \\
 -2k&= d(S^3_{-2n,2n}(-2D_{k},2D_{k}),\spinc_n)
 - d(S^3_{-2n}(-2D_{k}),\spinc_{-n})-\epsilon_2 \label{deq3} \\
 0&= d(S^3_{-2n,2n}(-U,U),\spinc_n)
 -d(S^3_{-2n}(-U),\spinc_{-n})-\epsilon_2, \label{deq4}
\end{align}
where $\spinc_0$ is the unique spin structure on $S^3$ and $\epsilon_i$ are grading shifts. The grading shifts $\epsilon_i$ are homological invariants \cite{os2} and hence (\ref{deq1}) and (\ref{deq3}), respectively, have the same grading shifts $\epsilon_1$ and $\epsilon_2$ as (\ref{deq2}) and (\ref{deq4}), respectively.
\end{theorem}

\begin{proof}
Ozsv\'{a}th and Szab\'o~\cite[Corollary 4.2]{Knots} showed that  for a knot $K$ in $S^3$ and $|- 2n| \ge 2g(K) - 1$, the complex $CF^+(S^3_{-2n}(K),\spin_{-n})$ is filtered chain homotopic to
$$
CFK^\infty(S^3,K)/CFK^\infty(S^3,K)_{\{i<0\,\cup\, j<-n\}}[\epsilon],
$$
where the grading shift $\epsilon$ is $d(S^3,\spinc_0)+\epsilon_1$ according to \cite{os2}.  (The value of $\epsilon_1$ can be computed explicitly, 
but we do not need its exact value in our computations.)

Proposition~\ref{prop:DandT23} along with Proposition~\ref{acyclicprop}  allows us to replace $-2D_k$ with $-T_{2,4k+1}$. We see that in the complex $CFK^\infty(S^3,-T_{2,4k+1})$ the cycle at filtration level $(0,-2k)$ is the cycle of grading zero having the least $j$--filtration among all grading zero cycles, and all cycles of grading less than zero have $i$--filtration less than zero. Since $-2k>-n$, the cycle at filtration level $(0,-2k)$ lives and all the cycles of grading less than zero vanish in the quotient.
See Figure~\ref{fig:cfk_quotient1} for the case $-n=-4$ and $k=1$.

\begin{figure}[h]
\fig{1}{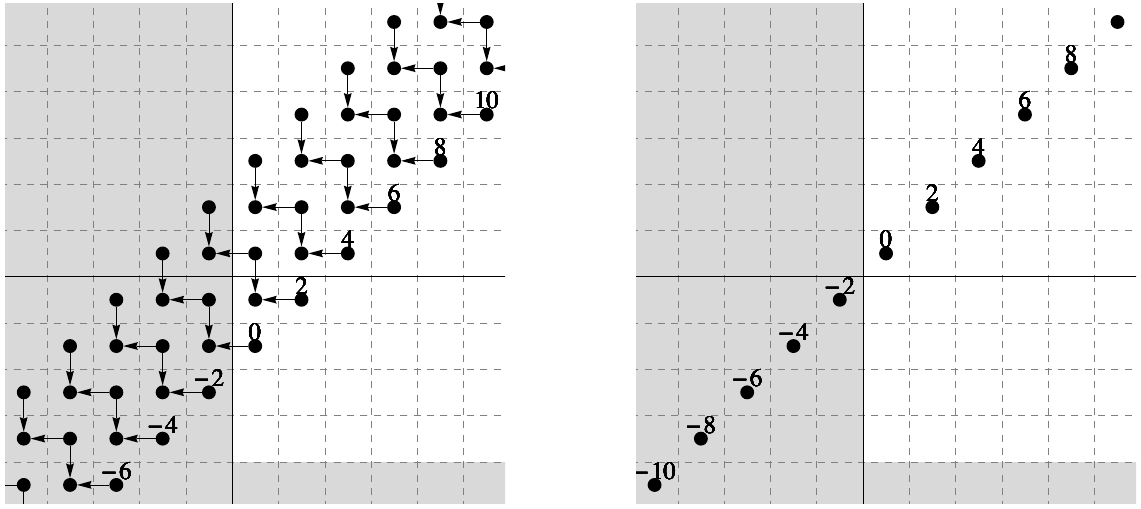}
\caption{}\label{fig:cfk_quotient1}
\end{figure}

This shows the identity $d(S^3_{-2n}(-2D_{k}),\spinc_{-n}) -d(S^3,\spinc_0) -\epsilon_1=0$.  A similar argument shows $d(S^3_{-2n}(-U),\spinc_{-n}) -d(S^3,\spinc_0) -\epsilon_1=0$.  These two identities give rise to Equation~\ref{deq1}.  Equation~\ref{deq2} is similar.

By Theorem~\ref{thmt24k+1}, noting $ n>2k$, we can identify
\[
CFK^\infty(S^3_{-2n}(-2D_{k}),\mu\#2D_{k},\spinc_n)_{i,j}=
CFK^\infty(S^3,T_{2,4k+1})_{i,j+1}\oplus A,
\]
where $A$ is an acyclic complex. 
That is, the  complex $CFK^\infty(S^3_{-2n}(-2D_k),2D_k,\spinc_n)$ is filtered chain homotopic to $CFK^\infty(S^3,T_{2,4k+1})$ with $j$--filtration shifted downward by one plus an acyclic complex $A$.
Combining this with Theorem~\ref{dsurgthm} and using Proposition~\ref{acyclicprop} to eliminate the acyclic summand from the computation, we have that the $d$--invariant associated to $CF^+(S^3_{-2n,2n}(-2D_{k},2D_{k}),\spinc_n)$ is equal to  that of 
\[
CFK^\infty(S^3,T_{2,4k+1})_{i,j+1}/ CFK^\infty(S^3,T_{2,4k+1})_{i,j+1  \{i<0,\,j<0\}}[\epsilon],
\]
where   $\epsilon=d(S^3_{-2n}(-2D_k),\spinc_n)+ \epsilon_2$ for some $\epsilon_2$ independent of $T_{2,4k+1}$.
The cycles $x$ at filtration level $(i,2k-i-1)$, $0\le i\le 2k$, are all of the
grading zero cycles in $CFK^\infty(S^3,T_{2,4k+1})_{i,j+1}$. It is easy to see
that the cycles $U^{k}x$ at filtration level $(i',-i'-1)$, $-k\le i'\le k$,
have grading $-2k$ and none of them vanish in the quotient, while
at least one of $U^{k'}x$ vanishes in the quotient if $k'>k$.
See Figure~\ref{fig:cfk_quotient2} for the case $n=4$ and $k=1$.

\begin{figure}[h]
\fig{0.93}{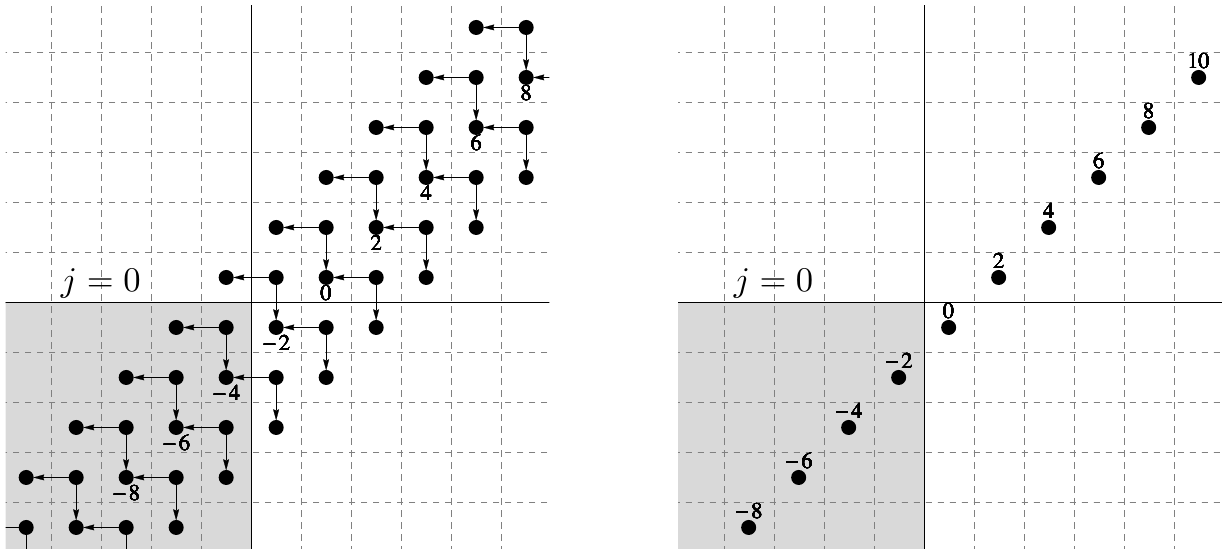} 
 \caption{}\label{fig:cfk_quotient2}
\end{figure}

This implies Equation~\ref{deq3}.
Combining Theorems~\ref{thmsuu} and \ref{dsurgthm}, a similar argument as done
above provides the proof of Equation~\ref{deq4}.

\end{proof}

Combining the equations in Theorem~\ref{diff_of_d_thm} immediately yields the following proposition, which was the key step in the completion of the proof of Theorem~\ref{main1} at the end of  Section~\ref{sectionexamples}.

\begin{prop} \label{prop:dmn}
$d(M(K_{D_{k},n}),\spinc_n)-d(M(K_{U,n}),\spinc_n)=-2k$.
\end{prop}\qed
 \vskip.01in

%%%%%%%APPENDIX%%%%%%%%%%%%%%
%%%%%%%APPENDIX%%%%%%%%%%%%%%
%%%%%%%APPENDIX%%%%%%%%%%%%%%

\begin{appendix}%zzz

%%%%%%%SECTION%%%%%%%%%%%%%%
\section{The infinity complex of the Whitehead doubled trefoil}\label{appendixcfkd}
Let $D$ denote the positive-clasped untwisted Whitehead double of the right-handed trefoil.  In this  appendix we prove:\vskip.1in

\noindent{\bf Proposition~\ref{prop:DandT23}.} {\it 
 The chain complex $\cfk^\infty(S^3,D)$ is chain homotopy equivalent to
 the chain complex $\cfk^\infty(S^3,T_{2,3})\oplus A$, where $A$ is an acyclic
 complex. The presence of the acyclic summand does not change the %relevant 
 width:  
 $$\Wd( \cfk^\infty(S^3,D))  = \Wd(\cfk^\infty(S^3,T_{2,3})) .$$}\vskip.1in

In order to prove this proposition, we need the following well-known lemma
about how a basis change affects the two-dimensional diagram of a knot Floer
complex. See \cite[Lemma~6.1]{hedden1} for instance.
\begin{lemma} \label{lemma:basischange}
Let $C_\ast$ be a knot Floer complex with a two-dimensional arrow diagram
 $D$ given by an $\ff$--basis.
 Suppose that $x,y$ are two basis elements of the same grading such that each
 of the $i$ and $j$ filtrations of $x$ is not greater than that of $y$. Then
 the basis change given by $y'=y+x$ gives rise to a diagram $D'$ of $C_\ast$
 which differs from $D$ only at $y$ and $x$ as follows:
\begin{itemize}
 \item Every arrow from some $z$ to $y$ in $D$ adds an arrow from $z$ to
   $x$ in $D'$.
  \item Every arrow from $x$ to some $w$ in $D$ adds an arrow from $y'$ to
   $w$ in $D'$.
 \end{itemize}
\end{lemma}
\begin{proof}
 First note that this basis change does not alter the grading or double
 filtrations.
 If $\partial z=y+\alpha$ for $z,\alpha \in C_\ast$, then $\partial z=y'+x+\alpha$, which
 shows that every arrow from $z$ to $y$ should add an arrow from $z$ to
 $x$. Since $\partial y'=\partial y+\partial x$, every arrow from $x$
 should add an arrow from $y'$.
 See Figure~\ref{figure:cfkhat_basechange} for an example.
\end{proof}

\begin{figure}[h]
\fig{.5}{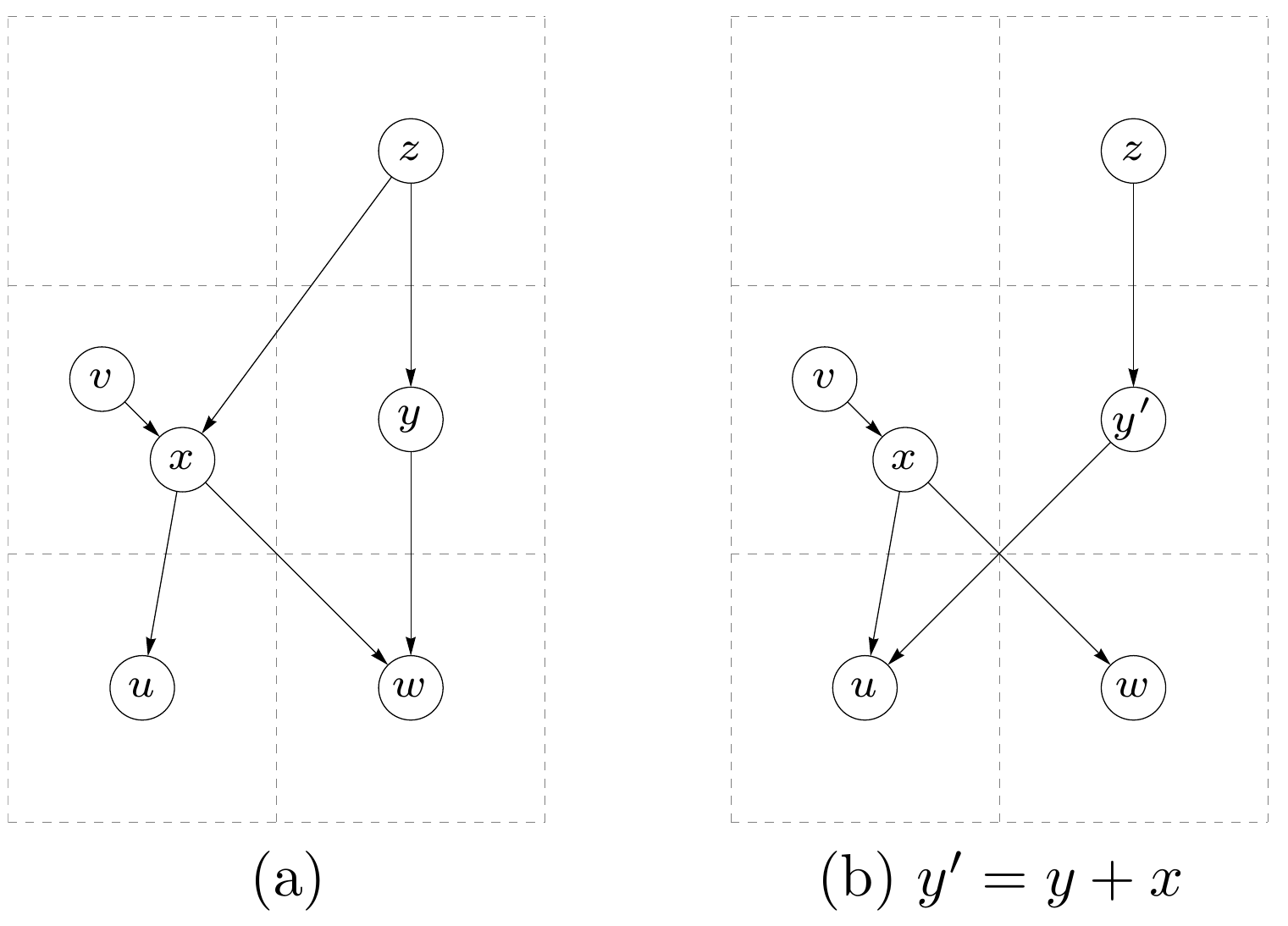}
\caption{The figure represents the effect of a filtered basis change to a portion of a $\Z\oplus\Z$ filtered chain complex over $\F$.}\label{figure:cfkhat_basechange}
\end{figure}
 
\begin{proof}[Proof of Proposition~\ref{prop:DandT23}]
Theorem~1.2 of \cite{hedden1} shows that
 \[
 \widehat{\hfk}_\ast(D,j)=
 \begin{cases}
  \ff^2_{(-1)}\oplus \ff^2_{(0)}, & j=1 \\
  \ff^4_{(-2)}\oplus \ff^3_{(-1)}, & j=0 \\
  \ff^2_{(-3)}\oplus \ff^2_{(-2)}, & j=-1 \\
  0, & \text{otherwise}.
 \end{cases}
 \]
 We assign $\ff$--bases to each summand in the direct sum decomposition as follows:
 \[
 \widehat{\hfk}_\ast(D,j)=
 \begin{cases}
  \langle u_1,u_2\rangle\oplus \langle x_1,x_2\rangle, & j=1 \\
  \langle v_1,v_2,v_3,v_4\rangle\oplus \langle y_1,y_2,y_3\rangle, & j=0 \\
  \langle w_1,w_2\rangle\oplus \langle z_1,z_2\rangle, & j=-1.
 \end{cases}
 \]
 Following Rasmussen~\cite[Lemma~4.5]{ras:thesis} (or
 \cite[Lemma~5.3]{hedden1}), $\widehat{\hfk}_\ast(D)$
 is chain homotopy equivalent to   $\widehat{\cfk}(D)$.
 So we assume that $\cfk^\infty(D)_{0,j}=\widehat{\hfk}_\ast(D,j)$ and 
 $\cfk^\infty(D)_{i,j}\cong U^{-i}\cfk^\infty(D)_{0,j-i}
 =\widehat{\hfk}_{\ast-2i}(D,j-i)$.
 If necessary, we put the grading in the superscript of
 the generator; for instance, $x_1^2$ denotes the grading 2 generator among
 $U^i x_1$ for $i\in \zz$.  See Figure~\ref{figure:cfkhat_D} for an example.

 First note that there are no components of boundary maps between
 generators of the same $(i,j)$--filtration since they would be reduced in
 $\widehat{\hfk}_\ast(D,j)$.  If we denote the vertical, horizontal, and
 diagonal components of the boundary map $\partial$ of $\cfk^\infty(D)$ by
 $\partial_V$, $\partial_H$, and $\partial_D$, respectively, then
 $\partial=\partial_V+\partial_H+\partial_D$.  We will determine $\partial$
 by first determining  $\partial_V$, then $\partial_H$,  and  lastly $\partial_D$.

 Note that $\ff^2_{(0)} \stackrel{\partial_V}{\to} \ff^3_{(-1)}
 \stackrel{\partial_V}{\to} \ff^2_{(-2)}$, or, $\langle x_1,x_2\rangle
 \stackrel{\partial_V}{\to} \langle y_1,y_2,y_3\rangle
 \stackrel{\partial_V}{\to} \langle z_1,z_2\rangle$ is a chain subcomplex
 of $\widehat{\cfk}(D)$ since $\partial$ lowers the grading by one. Since
 $\widehat{\hf}(S^3)=\ff_{(0)}$, by changing basis we  may assume that $\partial_V(x_1)=\partial_V(y_1)=
 \partial_V(z_1)= \partial_V(z_2)=0$, $\partial_V(x_2)=y_1$, $\partial_V(y_2)=z_1$, and $\partial_V(y_3)=z_2$.
 See Figure~\ref{figure:cfkhat_D}(b).
 
\begin{figure}[h]
\fig{.5}{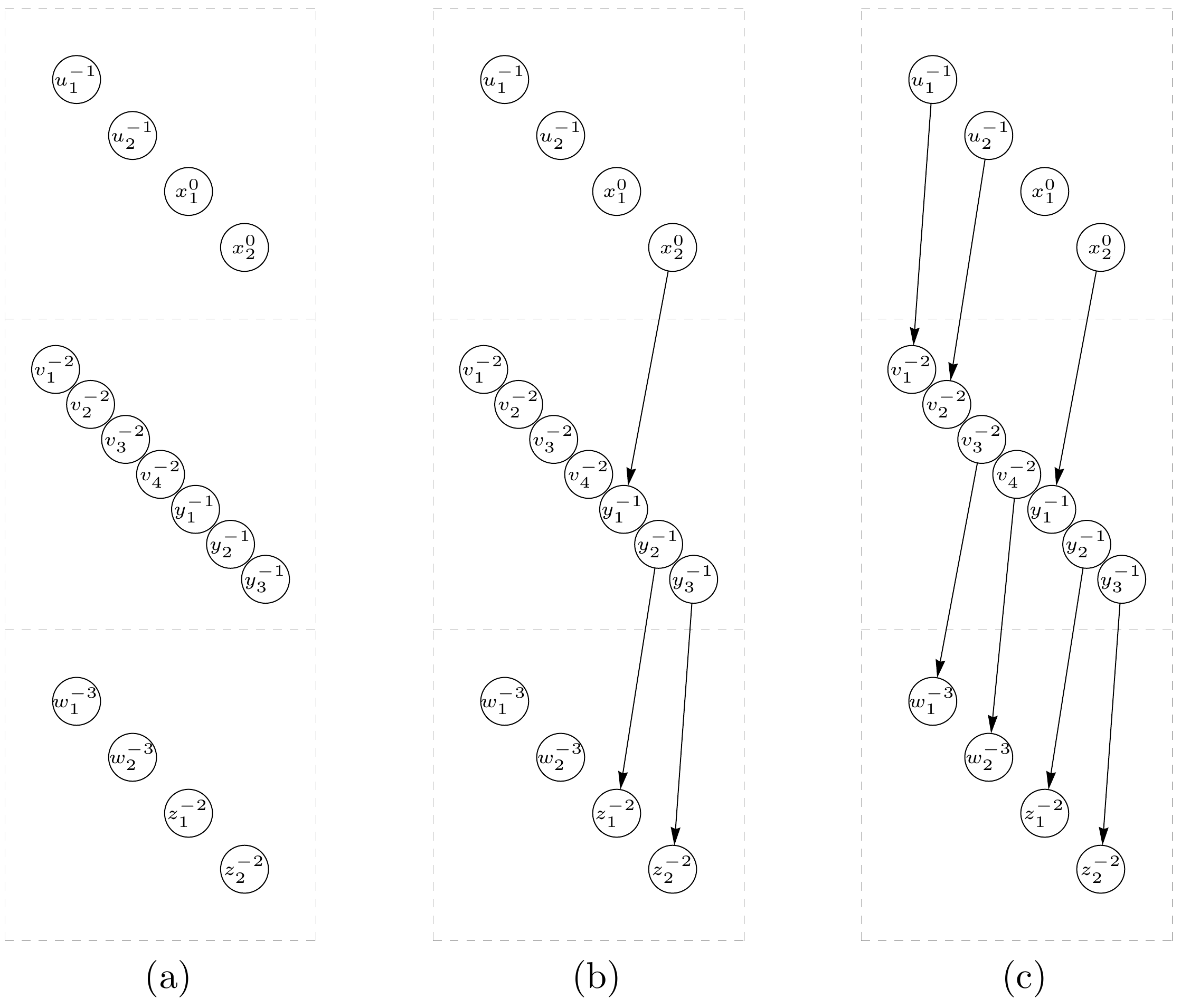}
\caption{}\label{figure:cfkhat_D}
\end{figure}

 We will find $\partial_V(u_1)$, which must lie in $\langle
 v_1,v_2,v_3,v_4,z_1,z_2\rangle$. If $\partial_V(u_1)=az_1+bz_2\in \langle
 z_1,z_2\rangle$ for $a,b\in\ff$, then $u_1+ay_1+by_2$ represents a
 nontrivial element of grading $-1$ in $\widehat{HF}(S^3)$, which is
 impossible. Thus $\partial_V(u_1)$ must have a nontrivial component in
 $\langle v_1,v_2,v_3,v_4\rangle$, which may be assumed to be $v_1$ by
 changing the basis for $\langle v_1,v_2,v_3,v_4\rangle$. If
 $\partial_V(u_1)=v_1+az_1+bz_2$, then the change of basis $v'_1=v_1+az_1+bz_2$
 gives rise to $\partial_V(u_1)=v'_1$ and $\partial_V v'_1=\partial_V v_1$,
 as in Lemma~\ref{lemma:basischange}. So we may assume that
 $\partial_V u_1=v_1$ and  similarly that $\partial_V u_2=v_2$.
 The image of $\langle v_3,v_4\rangle$ under $\partial_V$ should be equal
 to $\langle w_1,w_2\rangle$ since $\widehat{HF}(S^3)=\zz$ in which
 $v_3,v_4,w_1$ and $w_2$ should vanish.
 So $\{\partial_V(v_3), \partial_V(v_4)\}$ is a basis for $\langle
 w_1,w_2\rangle$ and we may assume that $w_1=\partial_V(v_3)$ and
 $w_2=\partial_V(v_4)$. The vertical components of the boundary maps are
 all determined as shown in Figure~\ref{figure:cfkhat_D}(c).

 Next, we will determine the horizontal components of the boundary map of
 $\cfk^\infty(D)$, whose columns look like
 Figure~\ref{figure:cfkhat_D}. We will argue that the complex will have a
 two-dimensional illustration described in Figure~\ref{figure:cfk_D}.
By analogy with the vertical case, note that $\langle z_1,z_2\rangle
 \stackrel{\partial_H}{\to} \langle y_1,y_2,y_3\rangle
 \stackrel{\partial_H}{\to} \langle x_1,x_2\rangle$ is a chain subcomplex
 $S$ of $\cfk^\infty(D)_{\{j\le 0\}}/\cfk^\infty(D)_{\{j<0\}}$, since
 $\partial$ lowers the degree by one.  Observe as well that for any $s \in
 S$, elements with grading one lower than $s$ are either to the left or below and
 hence $\partial s =\partial_V s+\partial_H s$. In particular there are no
 diagonal components of the boundary maps restricted to $S$. This implies
 that $\partial x_1=\partial_V x_1=0$ and $\partial x_2=\partial_V
 x_2=y_1$.

 Since $\widehat{HF}(S^3)\cong\ff_{(0)}$ is isomorphic to
 $H_\ast\left(\cfk^\infty(D)_{\{j\le 0\}}/\cfk^\infty(D)_{\{j<0\}}\right)$,
 we   may choose an $\ff$--basis $\{z_1,z_2\}$ so that $\partial_H(z_1)=0$. To
 keep the same vertical description as in Figure~\ref{figure:cfkhat_D}(c),
 we adjust the basis for $\langle y_2,y_3\rangle$ accordingly.
 Observe that $\partial z_2$ is the source of  no diagonal arrows, since elements with grading one lower  are located only to the left. So we have $\partial z_2\in \langle
 y_1,y_2,y_3\rangle$. If $\partial z_2$ is of the form $y_2+\beta$ for $\beta \in\langle
 y_1,y_3\rangle$, then $0=\partial^2 z_2=\partial y_2+\partial
 \beta =z_1+\partial_H y_2+\partial \beta $, which, on the other hand,
 can never be zero since $\partial_H y_2\in\langle x_1,x_2\rangle$, $\partial
 \beta \in\langle \partial y_1,\partial y_3\rangle \in \langle z_2,x_1,x_2\rangle$, and
 $z_1$ does not belong to $\langle x_1,x_2,z_2\rangle$. Thus $y_2$ does not appear in $\partial z_2$. Similarly,  $y_3$  does not appear in $\partial z_2$, and thus
 $\partial z_2$ must be $y_1$.

 Similarly, grading considerations and the fact that the homology of the quotient $\cfk^\infty(D)_{\{j\le m\}}/\cfk^\infty(D)_{\{j<m\}}$ is $\F_{(2m)}$  implies that $\partial_H\langle y_2,y_3\rangle = \langle
 x_1,x_2\rangle$. If $\partial_H y_2$ is of the form $x_2+ax_1$ for
 $a\in\ff$, then $\partial y_2=(\partial_H +\partial_V)y_2=x_2+ax_1+z_1$
 and $0=\partial^2 y_2=\partial (x_2+a x_1+z_1)=y_1\ne 0$, which is
 impossible. Thus we have $\partial_H y_2=x_1$. Then $\partial_H y_3$ should
 be of the form $x_2+ax_1$. By the change of  basis $x_2'=x_2+ax_1$ we may assume
 $\partial_H y_3=x_2$.

 To complete the analysis of $\partial_H$ we must consider the $w, v $ and $ u$ generators.  
   The argument is similar to what we have done already.  
First notice that these elements might not generate a subcomplex of the horizontal complex; $\partial_H(w_i^j) $ could contain  terms of the form $x_k^{j-1}$ (which are at the same $j$--filtration level  but at $i$--filtration two lower).  A change of basis, adding some of the elements  $x_k^{j-1}$ to some of the $v_i^{j-1}$, eliminates this possibility, at the expense of perhaps adding diagonal maps.  Since the change of basis combines elements at different $i$--filtration levels, the vertical map is unchanged.   Thus, we can assume that the $w, v$ and $u$ generate a subcomplex of the horizontal complex which is complementary to the subcomplex generated by the $x, y$, and $z$ generators.
 Using the fact that $\partial^2(u_k^{i-1}) = 0$ we conclude that $\partial_H$ must vanish on the $v_1^i$ and $v_2^i$.   
 
 Using the known  homology of the horizontal complex (in particular, that the horizontal homology at $j$--filtration level $0$ is generated by a single element at grading 0, and thus a $z_i$) we can conclude that $\partial_H$ maps the subgroup generated by $v_3^0$ and $v_4^0$   isomorphically to the subgroup generated by $u_1^{-1}$ and $u_2^{-1}$, and similarly for their $U$ translates.  

   A change of basis among the $u_1$ and $u_2$ generators ensures each $v_3$ maps to the corresponding $u_1$ and each $v_4$ maps to a corresponding $u_2$.   A change of basis among the $v_1$ and $v_2$ elements reestablish that $\partial_V$ maps each $u_1^i$ and $u_2^i$ to a $v_1^{i-1} $ and $v_2^{i-1}$, respectively.  That $\partial^2(u_k^i) = 0$ implies that $w_i^j$ maps horizontally to a corresponding $v_i^{j-1}$.
 
 At this point we have a diagram for $\cfk^\infty(D)$ as in
 Figure~\ref{figure:cfk_D} with only vertical and horizontal components of
 the boundary maps shown.

\begin{figure}[h]
\fig{.42}{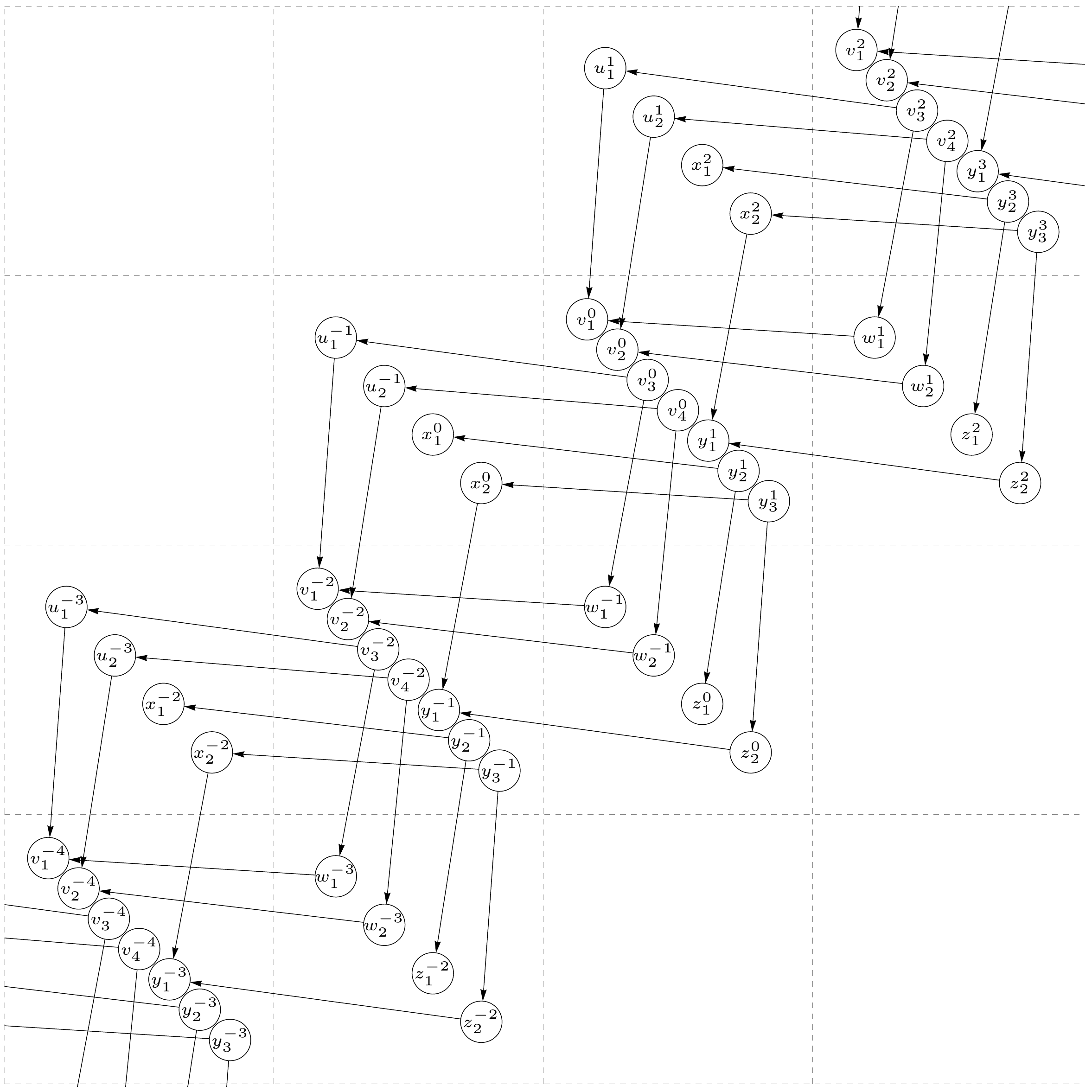} \\
\caption{}\label{figure:cfk_D}
\end{figure}
 
Finally, we will deal with the diagonal components of the boundary maps.
As mentioned earlier, due to grading constraints there are no diagonal maps
coming from the $x$, $y$, or $z$ generators,  while there may be diagonals
going in. On the other hand, there are no diagonal maps going into the $u$, $v$, or $w$ generators. All possible cases of diagonal maps
are: (1) from $u$'s to $x$'s, (2) from $v$'s to $y$'s, and (3) from $w$'s
to $z$'s.  This implies that the complex $T$ generated by $x_1$, $y_2$ and
$z_1$ is indeed a subcomplex of $\cfk^\infty(D)$.

We will show that filtered basis changes can eliminate all the diagonal
arrows going into $T$. Then $\cfk^\infty(D)$ splits into $\ff[U,U^{-1}]\otimes T$ and a
subcomplex $A$. Note that $\ff[U,U^{-1}]\otimes T$ is isomorphic to
$\cfk^\infty(T(2,3))$ and $A$ is acyclic (that is, $H_\ast(A)=0$). This follows from \cite[Section 10]{oz:hf-properties}, which showed $HF^\infty(S^3)\cong
 \ff[U,U^{-1}]$ and $HF^\infty(S^3)\cong H_\ast(\cfk^\infty(D))\cong
 H_\ast(\ff[U,U^{-1}]\otimes T)$ as $\ff[U,U^{-1}]$--modules.

 First, we show that $\partial v_1$ and $\partial v_2$ cannot include
 $y_2$. Note that they have zero vertical and horizontal components. If
 $\partial v_1=y_2+ay_1+by_3$ for $a,b\in\ff$, then $0=\partial^2
 v_1=x_1+z_1+bx_2+bz_2$ which cannot be zero for any $a,b$. So there are no
 arrows from $v_1$ or $v_2$ to $y_2$.

 We claim that, for any $a,b\in\ff$ and $i=1,2$,
 the following are equivalent:
 \begin{enumerate}
\item $\partial_D u_i=ax_1+bx_2$
\item $\partial_D v_{i+2}=ay_2+by_3+cy_1$ for some $c\in\ff$
\item $\partial_D w_i=az_1+bz_2$.
 \end{enumerate}
 
 We prove the claim only for $i=1$; almost the same argument applies to $i=2$.
 Let $\partial_D u_1 =ax_1+bx_2$, $\partial_D v_{3} =
 c_1y_1+c_2y_2+c_3y_3$, and $\partial_D w_1= d_1z_1+d_2z_2$ for
 $a,b,c_\ast, d_\ast\in\ff$. The constraint $\partial^2=0$ gives rise to
 the equalities
 \begin{align*} 
  0 &= \partial^2 v_3=\partial(u_1+w_1+c_1y_1+c_2y_2 +c_3y_3)\\
   &= (v_1+ax_1 +
  bx_2)+ (v_1+d_1z_1 + d_2z_2)+ c_2 (x_1+z_1) + c_3(x_2+z_2) \\
  &= (a+c_2)x_1+(b+c_3)x_2+(d_1+c_2)z_1+(d_2+c_3)z_2.
 \end{align*}
 Thus $a=c_2=d_1$ and $b=c_3=d_2$.
 
 Suppose $a=1$ for $i=1$. Let $v_1'=v_1+x_1$ and $w_1'=w_1+y_2$. Then
 all arrows going into $v_1$ or $w_1$ come from $u_1$, $w_1$ or $v_3$ and
 hence we need to check the boundaries of $u_1$, $v_3$, $w'_1$ and $v'_1$:
  $\partial u_1=v'_1+bx_2$, $\partial v_3=u_1+w'_1+by_3+cy_1$, $\partial
 w'_1=\partial w_1+\partial y_2=(v_1+z_1+bz_2)+(x_1+z_1)=v_1+bz_2$, and
 $\partial v'_1=\partial v_1+\partial x_1=0$. With these new basis elements
 $v'_1$ and $w'_1$ there are no diagonal components from $u_1, v_3, w_1$ to
 $x_1,y_2,z_1$.
 A similar argument works for $u_2,v_4,w_2$. Thus $T$ can be assumed to be a
 direct summand as desired.
\end{proof}
We remark that a similar process of changing bases as in the previous proof can be used to prove that $\cfk^\infty(S^3,D)$ is isomorphic to the complex in 
Figure~\ref{figure:cfk_D}. Since this result is unnecessary for our purposes or any foreseeable applications to concordance we leave it as an exercise for the interested reader.

%%%%%%%%%%%%%APPENDIX SECTION %%%%%%
 %%%%%%%SECTION%%%%%%%%%%%%%%

\section{ $CFK^\infty(S^3, T_{2,2k+1})$} \label{appendixt22k+1}
\begin{theorem}\label{appendixt22k+1thm} $CFK^\infty(S^3 , T_{2,3})^{\otimes k} = CFK^\infty(S^3, T_{2,2k +1}) \oplus A$ where $A$ is acyclic.  The presence of the acyclic summand does not change the %relevant 
width:  $$\Wd(CFK^\infty(S^3, T_{2,2k +1}))  = \Wd(CFK^\infty(S^3, T_{2,2k +1})).$$

\end{theorem}

\begin{proof} The proof is by induction.  We show that $$CFK^\infty(S^3, T_{2,2k +1}) \otimes CFK^\infty(S^3 , T_{2,3}) =  CFK^\infty(S^3, T_{2,2k +3}) \oplus A.$$

The complex $ CFK^\infty(S^3, T_{2,2k +1}) $ has filtered generators at grading 0: $[x,i,j]$ where $i\ge 0$, $j\ge 0$ and $i+j = k$.  There are also generators at grading level 1, $[y,i,j]$ with $i\ge 1, j\ge 1$ and $i+j = k+1$.  The boundary map is given by $\partial [y,i,j] = [x,i-1,j] + [x, i,j-1]$.  (Notice that the symbols $x$ and $y$ do not correspond to intersection points in a Heegaard diagram.  The $i$ and $j$ denote the filtration levels.)

In order to distinguish the complex for $T_{2,3}$, we replace $x$ and $y$ with $z$ and $w$, so that the complex is generated by $[z,0,1], [z,1,0], $ and $[w,1,1]$.

The tensor product $CFK^\infty(S^3, T_{2,2k +1}) \otimes CFK^\infty(S^3 , T_{2,3})$ has generators of type $x \otimes z$ at grading level 0, $x\otimes w$ and $y \otimes z$ at grading level 1, and $y \otimes w$ at grading level 2.  In total there are $3(2k+1)$ generators. 

We now make a basis change, replacing certain generators with their sums with other generators, relabeled as indicated:

\begin{itemize}
\item $[x,i,j] \otimes [w,1,1] \to  [x,i,j] \otimes [w,1,1] + [y,i+1,j] \otimes [z, 0,1] = \alpha_i$, for all $0 \le i <k$.\vskip.1in
\item $[x,i,j] \otimes [z,1,0] \to  [x,i,j] \otimes [z,1,0] + [x,i+1,j-1] \otimes [z, 0,1] = \beta_i$, for all $0 \le i <k$.\vskip.1in
\item $[y,i,j] \otimes [z,1,0] \to  [y,i,j] \otimes [z,1,0] +[x,i,j-1] \otimes [w, 1,1] = \gamma_i$, for all $0 <  i \le k$.
\end{itemize}

Now we isolate out acyclic pieces, using the following four observations.

\begin{itemize}
\item $\partial ([y,i,j] \otimes [w,1,1]) = [x,i-1,j] \otimes [w,1,1] + [x,i,j-1] \otimes [w,1,1] + [y,i,j] \otimes [z,0,1] + [y,i,j] \otimes [z,1,0]  = \alpha_{i-1} + \gamma_i$.\vskip.1in

\item $\partial \alpha_{i-1} = \partial( [x,i-1,j] \otimes [w,1,1] + [y,i,j] \otimes [z, 0,1] )=[x,i-1,j] \otimes [z,0,1] + [x,i-1,j] \otimes [z,1,0] + [x,i-1,j] \otimes[z,0,1] + [x,i,j-1] \otimes [z,0,1] =
 [x,i-1,j] \otimes [z,1,0] +   [x,i,j-1]\otimes [z,0,1]= \beta_{i-1}$.\vskip.1in 
 
\item $\partial \gamma_i = [x,i-1,j] \otimes[z,1,0] + [x,i,j-1] \otimes [z,1,0] + [x,i,j-1]\otimes [z,0,1] + [x,i,j-1] \otimes[z,1,0] = 
[x,i-1,j] \otimes[z,1,0]  + [x,i,j-1] \otimes [z,0,1] = \beta_{i-1}$.\vskip.1in
\item $\partial \beta_{i-1} = 0$.  

\end{itemize}

From this we see that there is an acyclic {\it summand} $$\left<[y,i,j] \otimes [w,1,1] \right> \stackrel{\partial}{\to} \left< \alpha_{i-1},  \gamma_i \right> \stackrel{\partial}{\to} \left<\beta_i\right>.$$
For instance, see Figure~\ref{figuret25tensort23} for the case $k=2$.

\begin{figure}[h]
\fig{.5}{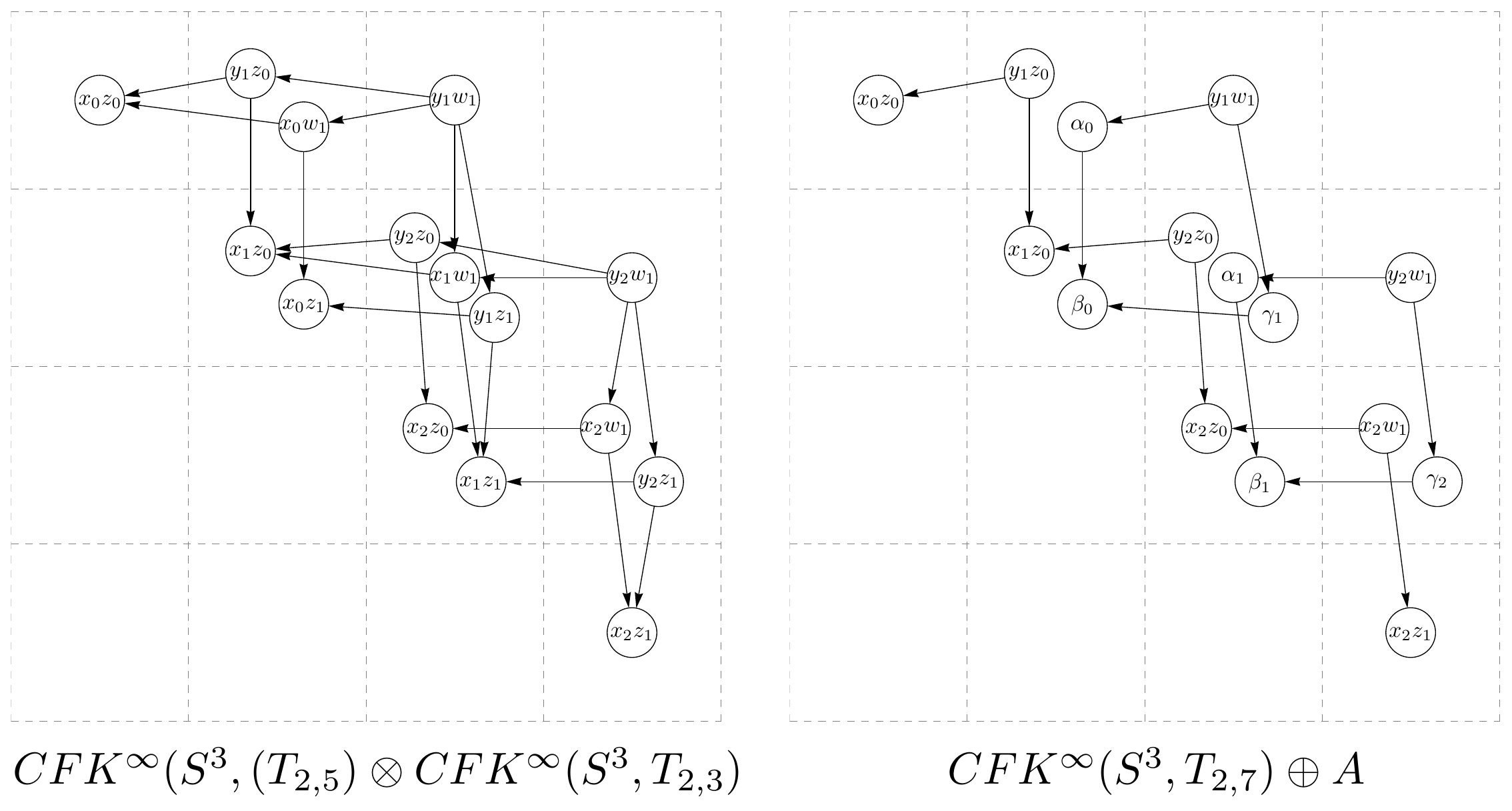}   \caption{
  Notation: $x_iz_{i'}=[x,i,k-i]\otimes [z,i',1-i']$,
  $x_iw_1=[x,i,k-i]\otimes[w,1,1]$, $y_iz_{i'}=[y,i,k+1-i]\otimes
 [z,i',1-i']$, $y_iw_1=[y,i,k+1-i]\otimes[w,1,1]$,
  $\alpha_i=x_iw_1+y_{i+1}z_0$, $\beta_i=x_iz_1+x_{i+1}z_0$, and
  $\gamma_i=y_iz_1+x_iw_1$.
  }
  \label{figuret25tensort23}
\end{figure}

There are $k$ such summands, with a total rank of $4k$.  The original complex had rank $3(2k+1) = 6k+3$.  Thus, splitting off the acyclic summands leaves a complex of rank $2k+3$.  Generators for a complement to the acyclic summand are given by the set    $\{ [x,i,j] \otimes [z,0,1],[y,i,j] \otimes[z,0,1] \}$ and two more elements, $[x,k,0]\otimes[w,1,1]$ and $[x,k,0]\otimes [z,1,0]$.  Finally, we note that this is a subcomplex of the desired isomorphism type, as follows from three simple observations:  $\partial ([x,i,j] \otimes [z,0,1]) = 0$, $\partial  ([y,i,j] \otimes [z,0,1]) = [x,i-1,j] \otimes [z,0,1]+ [x,i,j-1] \otimes [z,0,1]$ and $\partial ([x,k,0]\otimes [w,1,1]) = [x,k,0] \otimes [z,0,1] + [x,k,0] \otimes[z,1,0]$.
\end{proof}

Note that similar computations have recently appeared in \cite{Hom2}.
%%%%%%%%%%%%APPENDIX SECTION %%%%%%%%%
 %%%%SUBSECTION %%%%%%%%
\section{Number theoretic results}\label{appennumber}

\begin{theorem}
\label{thm:number}
There is an infinite set $\mathcal N$ of natural numbers $\{n_i\}$ satisfying:

\begin{enumerate}

\item  For all $n_i$, $ 4n_i^2 +1\ge 9 $ and   is either prime or the product of two distinct primes;  thus $ 4n_i^2 +1 $  is square free;\vskip.05in
\item The values $\{4n_i^2 +1\}$ are pairwise relatively prime. 

\end{enumerate}

\end{theorem}

\begin{proof}$\ $ A theorem of Iwaniec~\cite{iwaniec} states that  if $G(n) = an^2 + bn+c$ is an irreducible integer polynomial with $a >0$ and $c \equiv 1 \mod 2$, then there exist infinitely many $n$ such that $G(n)$ has at most two prime factors, counted with multiplicity.  
We will  apply this for $G(n)$ of the form $4A^2 n^2+1$, for appropriate values of $A >0$.  Notice that  any  $G(n)$ of this form is never a perfect square (for any $n >0$). Thus, by  Iwaniec's theorem we have that  for an infinite set of positive $n$,  $G(n)$ is either prime or a product of two distinct primes.  In particular, it is square free.

The $n_i$ are defined inductively, starting with  $n_1 = 2$, so $4n_1^2 +1 = 17$ is prime.    Suppose that for all $i <k$, values of $n_i$ have been selected so as to satisfy the conditions of the theorem.   Let $A$ denote the product of all   $4n_i^2 +1$, $i<k$.  Apply Iwaniec's theorem to choose an $N$ so that $4A^2N^2 +1$ is the product of at most two prime factors.  No prime factor of the $4n_i^2 + 1$, $i<k$, can divide this number, so $4A^2N^2 +1$ is relatively prime to  $4n_i^2 +1$ for all $i<k$.  Let $n_k = AN$.

\end{proof}

\end{appendix}

%%%%%%%BIBLIOGRAPHY%%%%%%%%%%%%%%

\vfill \eject
\newcommand{\etalchar}[1]{$^{#1}$}

\end{document}